\documentclass[smallcondensed,envcountsect,envcountsame]{svjour3}

\usepackage{calc}

\usepackage[utf8]{inputenc}
\usepackage[UKenglish]{babel}
\usepackage[english=british]{csquotes}

\usepackage{amsmath}

\newtheorem{hypothesis}[theorem]{Hypothesis}

\usepackage{xcolor}
\definecolor{HUblue}{rgb}{0,0.2157,0.4235}
\definecolor{HUred}{rgb}{0.5412,0.0588,0.0784}
\definecolor{HUsand}{rgb}{0.8235,0.7529,0.4039}
\definecolor{HUgreen}{rgb}{0,0.3412,0.1725}
\colorlet{structure}{HUblue}

\usepackage[T1]{fontenc}
\usepackage{lmodern}

\usepackage{hyperref}
\hypersetup{%
  colorlinks=true,%
  linkcolor=structure,%
  citecolor=structure,%
  urlcolor=structure,%
}

\usepackage{graphicx}

\usepackage{enumitem}

\usepackage[caption=false]{subfig}
\numberwithin{figure}{section}
\numberwithin{equation}{section}
\numberwithin{theorem}{section}
\numberwithin{table}{section}

\usepackage{amssymb}
\usepackage{mathabx}
\usepackage{bm}

\def\Xint#1{\mathchoice
{\XXint\displaystyle\textstyle{#1}}%
{\XXint\textstyle\scriptstyle{#1}}%
{\XXint\scriptstyle\scriptscriptstyle{#1}}%
{\XXint\scriptscriptstyle\scriptscriptstyle{#1}}%
\!\int}
\def\XXint#1#2#3{{\setbox0=\hbox{$#1{#2#3}{\int}$}
\vcenter{\hbox{$#2#3$}}\kern-.5\wd0}}
\newcommand{\intmean}{\Xint-}

\newlength{\raisebulletlen}
\setbox1=\hbox{$\bullet$}\setbox2=\hbox{\tiny$\bullet$}
\newcommand\pbullet{\raisebox{\raisebulletlen}{\,\scriptsize$\bullet$}\,}

\newcommand\Keyword[1]{{\bfseries #1}}

\newcommand\Do{\enskip\Keyword{do}}

\newcommand\Input{\noindent\Keyword{Input:}\enskip}

\newcommand\Od{\enskip\Keyword{od}}
\newcommand\Output{\noindent\Keyword{Output:}\enskip}

\newcommand\R{\mathbb R}
\newcommand\bbS{\mathbb S}
\newcommand\N{\mathbb N}

\newcommand\Ecal{\mathcal E}

\newcommand\Rcal{\mathcal R}

\newcommand\Tcal{\mathcal T}

\renewcommand\iff{\;\Leftrightarrow\;}
\DeclareMathOperator*\argmin{argmin}

\DeclareMathOperator\Div{div}

\DeclareMathOperator\grad{\nabla}
\DeclareMathOperator\gradNC{\grad_\NC}
\DeclareMathOperator\Grad{D\!}

\DeclareMathOperator\Hess{\operatorname{D}^2\!}
\DeclareMathOperator\id{id}

\let\inf\relax
\DeclareMathOperator*\inf{\vphantom{p}inf}

\DeclareMathOperator\Mid{mid}
\DeclareMathOperator\Res{\Rcal}
\DeclareMathOperator\sign{sign}

\newcommand\abs[1]{\vert #1 \vert}

\newcommand\dual[1]{\langle #1 \rangle}

\newcommand\Rsym\bbS

\newcommand\dx{\, \mathrm{d} x}
\newcommand\ds{\, \mathrm{d} s}

\newcommand\restrict[1]{\big\vert_{#1}}
\newcommand\sphere[1]{{\mathcal{S}(#1)}}

\newcommand\Dom{{\Omega}}

\newcommand\Tri{\Tcal}
\newcommand\Edges{\Ecal}

\newcommand\Marked{\mathcal{M}}

\newcommand\LSfun{{LS}}

\newcommand\Ltwo{{L^2(\Dom)}}

\newcommand\ndof{\texttt{ndof}}

\newcommand\Norm[2]{\Vert#1\Vert_{#2}}
\newcommand\NormEnergy[1]{\vvvert #1 \vvvert}

\newcommand\NormH[3]{\Norm{#1}{H^{#2}(#3)}}

\newcommand\NormHDiv[2]{\Norm{#1}{H(\Div,#2)}}
\newcommand\NormHDivDom[1]{\NormHDiv{#1}{\Omega}}
\newcommand\NormL[3]{\Norm{#1}{L^{#2}({#3})}}
\newcommand\NormLz[2]{\NormL{#1}{2}{#2}}

\newcommand\NormLzDom[1]{\NormLz{#1}{\Dom}}

\newcommand\Conf{{\textup{C}}}
\newcommand\CR{{\textup{CR}}}

\newcommand\Friedrichs{{\textup{F}}}

\newcommand\loc{{\textup{loc}}}
\newcommand\LS{{\textup{LS}}}
\newcommand\NC{{\textup{NC}}}

\newcommand\RT{{\textup{RT}}}

\newcommand\sym{{\textup{sym}}}

\usepackage{tikz}
\usepackage{pgfplots}

\usepackage{pgfplotstable}
\usepackage{booktabs}
\usepackage{array}

\usetikzlibrary{calc,shapes,external}

\tikzset{element/.style={
  scale=2,
  very thick,
  line join=round,
  line cap=round
}}
\tikzset{slopetriangle/.style={
  bottom color=black!20,
  middle color=black!5,
  top color=white,
  draw=black
}}
\tikzset{ref/.style={
  decoration={border,amplitude=5pt,segment length=4pt},
  preaction={decorate,draw,structure,semithick}
}}

\pgfplotsset{
  height=.40\textheight,
  grid style={densely dotted,semithick},
  legend style={
    legend columns=1,
    legend pos=outer north east,
    font=\footnotesize
  },
  compat=newest 
}

\newcommand\drawslopetriangle[3]{
  \pgfkeys{/pgf/fpu=true}
  \pgfmathparse{10*#2}\let\rightcoordinate\pgfmathresult;
  \pgfmathparse{10^(#1)*#3}\let\topcoordinate\pgfmathresult;
  \pgfkeys{/pgf/fpu=false}
  \path[draw] (axis cs: #2, #3) node[shape=coordinate] (BL) {};
  \path[draw] (axis cs: +\rightcoordinate, #3) node[shape=coordinate] (BR) {};
  \path[draw] (axis cs: #2, +\topcoordinate) node[shape=coordinate] (TL) {};
  \shadedraw[slopetriangle]
    (TL) -- (BL)
    node[midway, left] {\(#1\)}
    -- (BR)
    node[midway, below] {\(1\)}
    -- cycle;
}
\newcommand\drawswappedslopetriangle[3]{
  \pgfkeys{/pgf/fpu=true}
  \pgfmathparse{#2/10}\let\leftcoordinate\pgfmathresult;
  \pgfmathparse{10^(-#1)*#3}\let\bottomcoordinate\pgfmathresult;
  \pgfkeys{/pgf/fpu=false}
  \path[draw] (axis cs: #2, #3) node[shape=coordinate] (TR) {};
  \path[draw] (axis cs: +\leftcoordinate, #3) node[shape=coordinate] (BR) {};
  \path[draw] (axis cs: #2, +\bottomcoordinate) node[shape=coordinate] (TL) {};
  \shadedraw[slopetriangle]
    (BR) -- (TR)
    node[midway, above] {\(1\)}
    -- (TL)
    node[midway, right] {\(#1\)}
    -- cycle;
}

\pgfplotsset{
    colormap={parula}{
        rgb=(0.2081,0.1663,0.5292)
        rgb=(0.2116,0.1898,0.5777)
        rgb=(0.2123,0.2138,0.627)
        rgb=(0.2081,0.2386,0.6771)
        rgb=(0.1959,0.2645,0.7279)
        rgb=(0.1707,0.2919,0.7792)
        rgb=(0.1253,0.3242,0.8303)
        rgb=(0.0591,0.3598,0.8683)
        rgb=(0.0117,0.3875,0.882)
        rgb=(0.006,0.4086,0.8828)
        rgb=(0.0165,0.4266,0.8786)
        rgb=(0.0329,0.443,0.872)
        rgb=(0.0498,0.4586,0.8641)
        rgb=(0.0629,0.4737,0.8554)
        rgb=(0.0723,0.4887,0.8467)
        rgb=(0.0779,0.504,0.8384)
        rgb=(0.0793,0.52,0.8312)
        rgb=(0.0749,0.5375,0.8263)
        rgb=(0.0641,0.557,0.824)
        rgb=(0.0488,0.5772,0.8228)
        rgb=(0.0343,0.5966,0.8199)
        rgb=(0.0265,0.6137,0.8135)
        rgb=(0.0239,0.6287,0.8038)
        rgb=(0.0231,0.6418,0.7913)
        rgb=(0.0228,0.6535,0.7768)
        rgb=(0.0267,0.6642,0.7607)
        rgb=(0.0384,0.6743,0.7436)
        rgb=(0.059,0.6838,0.7254)
        rgb=(0.0843,0.6928,0.7062)
        rgb=(0.1133,0.7015,0.6859)
        rgb=(0.1453,0.7098,0.6646)
        rgb=(0.1801,0.7177,0.6424)
        rgb=(0.2178,0.725,0.6193)
        rgb=(0.2586,0.7317,0.5954)
        rgb=(0.3022,0.7376,0.5712)
        rgb=(0.3482,0.7424,0.5473)
        rgb=(0.3953,0.7459,0.5244)
        rgb=(0.442,0.7481,0.5033)
        rgb=(0.4871,0.7491,0.484)
        rgb=(0.53,0.7491,0.4661)
        rgb=(0.5709,0.7485,0.4494)
        rgb=(0.6099,0.7473,0.4337)
        rgb=(0.6473,0.7456,0.4188)
        rgb=(0.6834,0.7435,0.4044)
        rgb=(0.7184,0.7411,0.3905)
        rgb=(0.7525,0.7384,0.3768)
        rgb=(0.7858,0.7356,0.3633)
        rgb=(0.8185,0.7327,0.3498)
        rgb=(0.8507,0.7299,0.336)
        rgb=(0.8824,0.7274,0.3217)
        rgb=(0.9139,0.7258,0.3063)
        rgb=(0.945,0.7261,0.2886)
        rgb=(0.9739,0.7314,0.2666)
        rgb=(0.9938,0.7455,0.2403)
        rgb=(0.999,0.7653,0.2164)
        rgb=(0.9955,0.7861,0.1967)
        rgb=(0.988,0.8066,0.1794)
        rgb=(0.9789,0.8271,0.1633)
        rgb=(0.9697,0.8481,0.1475)
        rgb=(0.9626,0.8705,0.1309)
        rgb=(0.9589,0.8949,0.1132)
        rgb=(0.9598,0.9218,0.0948)
        rgb=(0.9661,0.9514,0.0755)
        rgb=(0.9763,0.9831,0.0538)
    }
}

\pgfplotscreateplotcyclelist{HU colors}{
  {HUblue,  every mark/.append style={fill=HUblue!60!white}},%
  {HUred,   every mark/.append style={fill=HUred!60!white}},%
  {HUsand,  every mark/.append style={fill=HUsand!60!white}},%
  {HUgreen, every mark/.append style={fill=HUgreen!60!white}},%
  {violet,  every mark/.append style={fill=violet!60!white}},%
  {cyan,    every mark/.append style={fill=cyan!60!white}}%
}
\pgfplotscreateplotcyclelist{HU markers}{
  {mark=square*},%
  {mark=*},%
  {mark=diamond*},%
  {mark=triangle},%
  {mark=otimes*}
}
\pgfplotscreateplotcyclelist{adaptive uniform}{
  {dashed, every mark/.append style={solid, fill=black!25!white}},
  {solid}
}

\pgfplotsset{convergenceplot/.style={
  xlabel=ndof,
  ylabel=error,
  ymajorgrids=true,
}}
\pgfplotsset{cases/.style={
  only marks,
  mark=,
  nodes near coords,
  nodes near coords align={anchor=center},
  point meta=explicit symbolic
}}
\pgfplotsset{efficiencyplot/.style={
  ymin=1,
  xlabel=ndof,
  ylabel=efficiency indices,
  ymajorgrids=true
}}
\pgfplotsset{triangulationplot/.style={
  axis equal image,
  enlargelimits={abs=1em},
  xtick={-1,-.5,...,1},
  ytick={-1,-.5,...,1}
}}
\pgfplotsset{patchstyle/.style={
  patch,
  white,
  faceted color=HUblue,
  line width=.5pt,
  table/row sep=\\
}}
\pgfplotsset{quiverplot/.style={
  axis equal image,
  enlargelimits={abs=1em},
  colorbar,
  xtick={-1,-.5,...,1},
  ytick={-1,-.5,...,1}
}}
\pgfplotsset{quiverstyle/.style={
  point meta=\thisrow{c},
  quiver={u={\thisrow{u}/\thisrow{c}},
          v={\thisrow{v}/\thisrow{c}},
          colored,scale arrows=.15,
          every arrow/.append style={line width=.5pt}
         },
  -stealth
}}
\pgfplotsset{pressureplot/.style={
  axis equal image,
  enlargelimits={abs=1em},
  colorbar,
  xtick={-1,-.5,...,1},
  ytick={-1,-.5,...,1}
}}

\newcommand\changed[1]{{\color{black} #1}}

\smartqed

\date{20th April 2017}

\title{Nonlinear discontinuous Petrov-Galerkin methods%
  \thanks{The work has been written while the first author enjoyed the
      hospitality of the Hausdorff Research Institute of Mathematics in
      Bonn, Germany, during the Hausdorff Trimester
      Program \enquote{Multiscale Problems: Algorithms, Numerical Analysis and
      Computation}. The second and third author were supported by the
      Berlin Mathematical School.  The research of all four authors has
      been supported by the Deutsche
      Forschungsgemeinschaft in the Priority Program 1748 \enquote{Reliable
      simulation techniques in solid mechanics. Development of non-standard
      discretization methods, mechanical and mathematical
      analysis} under the project \enquote{Foundation and application of
      generalized mixed FEM towards nonlinear problems in solid mechanics}
    (CA 151/22-1 and WR 19/51-1).}}
\titlerunning{Nonlinear dPG methods}%

\author{C.~Carstensen, P.~Bringmann, F.~Hellwig, P.~Wriggers}
\authorrunning{Carstensen et al.}
\institute{%
  C.\ Carstensen, P.\ Bringmann, F.\ Hellwig
  \at
    Humboldt-Universit\"at zu Berlin, Unter den Linden~6,
    10099~Berlin, Germany.\\
    \email{\{cc,bringman,hellwigf\}@math.hu-berlin.de}
  \and
  P.\ Wriggers%
  \at
    Gottfried Wilhelm Leibniz Universit\"at Hannover, Institut f\"ur
    Kontinuumsmechanik, Appelstraße~11, 30167~Hannover, Germany.\\
    \email{wriggers@ikm.uni-hannover.de}
}

\begin{document}

\maketitle

\begin{abstract}
  The discontinuous Petrov-Galerkin method is a minimal residual method
  with broken test spaces and is introduced for a nonlinear model problem in
  this paper. Its lowest-order version applies to a nonlinear uniformly
  convex model example and is equivalently characterized as a mixed
  formulation, a reduced formulation, and a weighted nonlinear
  least-squares method. Quasi-optimal a~priori and reliable and efficient
  a~posteriori estimates are obtained for the abstract nonlinear dPG
  framework for the approximation of a regular solution. The variational
  model example allows for a built-in guaranteed error control despite
  inexact solve. The subtle uniqueness of discrete minimizers is monitored
  in numerical examples.
\end{abstract}

\section{Introduction}
The discontinuous Petrov-Galerkin methodology (dPG) has recently been
introduced with the intention to design the optimal test spaces in a
Petrov-Galerkin scheme for maximal stability.
On the continuous level, the weak form of a PDE may assume the general form
$b(u,\cdot)=F$ with a unique solution $u$ in some real Banach space $X$ and
some bilinear form $b:X\times Y\to\R$ for some real Hilbert space $Y$ with
\changed{scalar product \(a: Y \times Y \to \R\) and}
a given right-hand side $F\in Y^*$, the dual to $Y$.
Well-posedness is understood to lead to an inf-sup condition on the
continuous level.
Given some discrete trial space $X_h\subset X$, the restriction
$b|_{X_h\times Y}$ clearly satisfies the inf-sup condition (even with a
possibly slightly better inf-sup constant) but it is less clear how to
choose the best trial space $M_h$, i.e. some subspace, $M_h\subset Y$ such
that
\begin{equation}\label{eq:intro_infsup}
0<\beta(X_h,M_h):= \inf_{x_h\in X_h} \sup_{y_h\in M_h} \frac{b(x_h,y_h)}{\Vert x_h \Vert_X \Vert y_h \Vert_Y}
\end{equation}
is maximal under the condition that $\dim(X_h) = \dim(M_h)$ is fixed.
The idealized dPG method computes the optimal test space utilizing some
Riesz representations in the infinite-dimensional Hilbert space
\(Y\) \cite{MR2837484}.
The practical realization utilizes, first, a test-search space $Y_h\subset Y$ with
dimension $n = \dim(Y_h)$ much larger than the dimension $m = \dim(X_h)$ of the trial space
$X_h$ and, second, a minimal residual method to compute the
discrete solution as a minimizer
\begin{equation}\label{eq:intro_min_res}
  x_h\in \argmin _{\xi_h\in X_h} \Vert F- b(\xi_h, \pbullet)\Vert_{Y_h^*}.
\end{equation}
The method is in fact equivalent to a Petrov-Galerkin scheme with the
bilinear form restricted to $X_h\times M_h$ for an appropriate subspace
$M_h \subset Y_h$ of dimension $m$ as pointed out in
\cite[Thm.~3.3]{MR3279489}.
Therefore, the large discrete space $Y_h$ (which is an input of the dPG scheme) is called
test-search space \cite{MR2916380} and the (implicit) test space \(M_h\) is not visible in \eqref{eq:intro_min_res}.

The computation of \(x_h\) in \eqref{eq:intro_min_res} is equivalent to
\changed{solving the normal equations and so possibly expansive.
This} guided Demkowicz and Gopalakrishnan \cite{MR3093480} to break the
norms in the test (and ansatz) spaces \cite{MR3521055}.
This allows a parallel computation of the dual norm separately  for each
individual element domain.
As it stands today, the term dPG abbreviates ``discontinuous
Petrov-Galerkin'' and stands for a {\em minimal residual method with broken test
or ansatz functions} and solely outlines a paradigm. The dPG
methodology allows various weak and ultra-weak formulations, where $X$ and
$Y$ are completely different and $b$ is not at all symmetric.
The least-squares finite element methods can be seen as a (degenerated)
subset of (an idealized) dPG with a degenerated test space in which the
Lebesgue norm can be evaluated exactly.

To the best knowledge of the authors, not much is known about nonlinear
versions of the methodology.
One first choice is to linearize the problem and then apply the dPG schemes
to the linear equations to generalize the Gauss-Newton method.
There exist already suggestions for nonlinear applications, in which there
are constraints plus a linear problem, e.g., for the contact problem in
\cite{1609.00765}.
\changed{Concepts of nonlinear dPG in fluid mechanics have been discussed
in \cite{MR3209958}.}
Another usage of the term {\em nonlinear} is in nonlinear approximation
theory and there is the contribution \cite{1511.04400} on linear problems with an attempt to
replace the Hilbert space $Y$ by some uniformly convex Banach space.

This paper introduces a direct {\em nonlinear dPG methodology} and replaces
the above bilinear form $b$ by some nonlinear mapping $b:X\times Y\to \R$,
which is linear and bounded in the second component to allow the
computation of the dual norm in the minimal residual method.
To stress the nonlinear dependence in the first component in $X$,
the notation in this papers follows \cite{MR731261} and separates
the linear components by a semi-colon so that the nonlinear dPG method replaces
$ b(\xi_h,\pbullet)$ in \eqref{eq:intro_min_res} by $ b(\xi_h;\pbullet)$.

The simplest case study for the nonlinear dPG methodology is an energy
minimization problem with some Hilbert space setting and a nonlinearity
with quadratic growth in the gradient.
The scalar model example of this paper stands for a larger class of Hencky
materials \cite[Sect.~62.8]{MR932255} and is the first model problem in
line towards real-life applications with a matrix-valued stress
$\sigma(F)$ given as a nonlinear function of some deformation gradient $F$
(such as the gradient $\nabla u$ of the displacement $u$) and the remaining
equilibration equation
\begin{equation}\label{eq:intro_model_problem}
f+\Div \sigma(\nabla u)=0 \quad\text{a.e. in }\Omega
\end{equation}
for some prescribed source term $f$ in the domain $\Omega$.
\changed{Although the existence of discrete solutions \(x_h\) to
\eqref{eq:intro_min_res} follows almost immediately, the closeness of
\(x_h\) to some continuous solution \(x\) is wildly open (cf.\
Remark~\ref{rem:kantorovich} below for a brief discussion).}

One critical point is the role of the stability condition
\eqref{eq:intro_infsup} in the nonlinear setting
for a regular solution and its low-order discretizations (as the most
natural first choice for nonlinear problems, partly because of limited
known regularity properties).
In  the situation of  the model scenario \eqref{eq:intro_model_problem},
the discrete stability follows from the stability of the continuous form
for piecewise constant $\nabla u_h$ and so the local discrete stability
simply follows from the linearization.

The overall structure of the nonlinear dPG of the type
\eqref{eq:intro_min_res} but for a nonlinear map $b$ with derivative $b'$
with respect to the first variable is also characterized as a nonlinear
mixed formulation with solution \((x_h,y_h)\in X_h\times Y_h\) to
\begin{equation*}
  \label{eq:mixed}
  \tag{M}
  \begin{aligned}
    a(y_h,\eta_h) + b(x_h; \eta_h)
    &= F(\eta_h)
    \quad\text{for all } \eta_h \in Y_h,
    \\
    b'(x_h; \xi_h, y_h)
    &= 0
    \quad\text{for all } \xi_h \in X_h.
  \end{aligned}
\end{equation*}
Another characterization in the lowest-order cases under consideration is
that as a weighted least-squares functional \changed{on Courant finite
element functions \(S^1_0(\Tri)\) with homogeneous Dirichlet boundary
values and the Raviart-Thomas finite element functions \(RT_0(\Tri)\)}
with some mesh-depending \changed{piecewise} constant weight \(S_0\in
P_0(\Tri;\R^{n\times n})\)
\begin{align*}
  (&u_\Conf, p_\RT)
  \in \argmin _{(v_\Conf,q_\RT) \in S^1_0(\Tri) \times RT_0(\Tri)}
  \Big(\Vert \Pi_0 f + \Div q_\RT \Vert_{L^2(\Omega)}^2\\
  &\phantom{{}\in{}}\quad+
  \Vert(I_{n\times n}+S_0)^{-1/2}
  \big(\Pi_0 q_\RT - \sigma(\nabla v_\Conf)
  + \Pi_0 (f(\id - \Mid(\Tri)))\big)\Vert_{L^(\Omega)}^2 \Big).
\end{align*}
This is already a new result even for the linear cases in
\cite{MR3576569,CP16a} and opens the door of a convergence analysis of
adaptive algorithms via a generalization of~\cite{MR3296614,CR16}.

This paper contributes the aforementioned  equivalent characterizations and
a first convergence analysis in the natural norms.
The a~priori result is local quasi-optimal convergence for the simple
model problem in that any discrete solution \(x_h\in X_h\) sufficiently
close to the exact regular solution \(x\in X\) satisfies
\[
  \Vert x-x_h \Vert_X
  \lesssim \inf_{\xi_h \in X_h} \Vert x-\xi_h \Vert_X.
\]

It has been discussed in \cite{MR3215064,MR3576569,CP16a} that the norm of
the computed residual \(\Vert y_h\Vert_Y = \Norm{F-b(v_C,q_\RT;\pbullet)}{Y_h^*}\) is almost a computable error estimator for linear problems
and this paper extends it to the a~posteriori error estimate
\begin{equation}\label{eq:intro_a_posteriori}
  \begin{aligned}
  \NormHDiv{ p- q_\RT}{\Omega}^2 + \NormEnergy{u-v_C}^2 &\approx \Norm{F-b(v_C,q_\RT;\pbullet)}{Y_h^*}^2\\
  &\quad+ \NormLzDom{(1-\Pi_0) f}^2 +  \NormLzDom{(1-\Pi_0)q_\RT}^2
  \end{aligned}
\end{equation}
for the nonlinear model problem \eqref{eq:intro_model_problem}. Since \(\Norm{F-b(v_C,q_\RT;\pbullet)}{Y_h^*}\) is the computable residual, this leads to built-in error control despite inexact solve: The discrete quantities \((v_C,q_\RT)\) in \eqref{eq:intro_a_posteriori} do \emph{not} need to solve the nonlinear dPG discrete problem.

The analysis is given for the primal version of the nonlinear dPG
\changed{for brevity} but applies to \changed{the} other formulations of
Subsection~\ref{sec:other_dpg} as well.
The results of this paper can be generalized, \changed{e.g., to the Hencky
material \cite[Sect.~62.8]{MR932255},} and then applied to more real life computational
challenges where the advantages of the dPG methodology are more striking.
\bigskip

The remaining parts of of this paper are organised as follows.
Section~\ref{sec:abstract_dpg} discusses an abstract framework for
different equivalent formulations of a dPG method for nonlinear problems
and develops an abstract a~priori estimate. Section~\ref{sec:model_problem}
presents a model problem with a dPG discretization.
Section~\ref{sec:analysis} analyses this discretization and gives proofs on
existence of a solution and an a~posteriori error estimate. Some numerical
examples in Section~\ref{sec:numerical_experiments} conclude the paper.
\bigskip

This paper employs standard notation of Sobolev and Lebesgue spaces
\(H^k(\Dom)\), \(H(\Div,\Dom)\), \(L^2(\Dom)\), and \(L^\infty(\Omega)\)
and the corresponding
spaces of vector- or matrix-valued functions \(H^k(\Dom;\R^n)\),
\(L^2(\Dom;\R^n)\), \(L^\infty(\Dom;\R^n)\), \(H^k(\Dom;\R^{n \times n})\),
\(H(\Div,\Dom;\R^{n \times n})\), \(L^2(\Dom;\R^{n \times n})\), and
\(L^\infty(\Dom;\R^{n \times n})\).
For any regular triangulation \(\Tri\) of \(\Omega\), let \(H^k(\Tri)
\coloneq \prod_{T \in \Tri} H^k(T) \coloneq \{v \in L^2(\Omega) \,|\,
\forall T \in \Tri,\, v\vert_T \in H^k(T)\}\) denote the piecewise
(or broken) Sobolev spaces and \((\gradNC v)\restrict{T} = \nabla (v\restrict{T})\) on \(T\in\Tri\) the piecewise gradient for \(v\in H^1(\Tri)\).
Let \(\vvvert{\pbullet}\vvvert \coloneq
\vert{\pbullet}\vert_{H^1(\Dom)} = \Vert{\grad
\pbullet}\Vert_{L^2(\Dom)}\) abbreviate the energy norm.
For every Hilbert space \(X\), let \((\pbullet, \pbullet)_X\) denote the
associated inner product and,
for every normed space \((X, \Vert \pbullet \Vert_X)\),
\(\sphere{X} \coloneq \{x \in X \,|\, \Vert x \Vert_X = 1\}\) the
sphere in \(X\).
The measure \(\vert \pbullet \vert\) is context-sensitive and refers to the
number of elements of some finite set or the length \(\vert E \vert\) of an
edge \(E\) or the area \(\vert T \vert\) of some triangle \(T\) and not just
the modulus of a real number or the Euclidean length of a vector.

Throughout the paper, \(A\lesssim B\) abbreviates the relation \(A \leq CB\)
with a generic constant \(0<C\), which does not depend on the mesh-size of
the underlying triangulation \(\Tri\) but
solely on the initial triangulation \(\Tri_0\); \(A\approx B\) abbreviates
\(A\lesssim B\lesssim A\), e.g. in \eqref{eq:intro_a_posteriori}.

\section{Abstract framework}\label{sec:abstract_dpg}
This section analyses an abstract nonlinear dPG method and presents an a~priori error estimate.
\subsection{Abstract nonlinear dPG}
For an open set \(D\neq \emptyset\) in a real Banach space \(X\) and a real
Hilbert space changed{\(Y\) with scalar product \(a: Y \times Y \to \R\)}, let
\(B\in C^1(D;Y^*)\) be a differentiable nonlinear map with Fr\'echet
derivative \(\Grad B(x)\in L(X;Y^*)\) at \(x \in D\).
With the duality bracket \(\langle \pbullet,\pbullet \rangle\) in \(Y\),
associate the nonlinear map \(b: X \times Y \to \R\),
\(b(x; \pbullet) \coloneq \langle B(x), \pbullet \rangle\), which is linear and bounded in
the second component. Let \(b'(x;\pbullet)\) abbreviate the derivative
\changed{\(\Grad B(x) \in L(X;Y^*)\)} with \(b'(x;\xi,\eta)\coloneq \langle
\Grad B(x;\xi),\eta\rangle\) for \(x\in D,\xi\in X,\eta \in Y\).
\bigskip

Given \(F \in Y^*\), let \(x \in D\) be a \emph{regular solution}
to the problem \(B(x) = F\) in \(Y^*\).
That means that \(x\) solves \(B(x) = F\) and the Fr\'echet derivative
\(\Grad B\) at \(x\) is a bijection from \(X\) to \(Y^*\).
The latter implies the inf-sup condition for the Fr\'echet derivative at the
regular solution \(x\), namely,
\begin{equation}
  \label{eq:continuous_inf_sup}
  0 < \beta(x)
  \coloneq
  \inf_{\xi \in \sphere{X}} \sup_{\eta \in \sphere{Y}} b'(x; \xi, \eta).
\end{equation}
The minimal residual formulation of the continuous problem seeks
\(x \in X\) with
\begin{equation}
  \label{eq:minres}
  x \in \argmin_{\xi \in D} \Vert F - B(\xi)\Vert_{Y^*}.
\end{equation}
The existence of a solution \(x\) to~\eqref{eq:minres} is immediate from
the assumption \(B(x) = F\).
In particular, the minimum is zero and any minimizer \(x\)
in~\eqref{eq:minres} solves \(B(x) = F\).
The situation is (in general) different on the discrete level with some
discrete subspaces \(X_h \subset X\) and \(Y_h \subset Y\), the
dPG scheme seeks a minimizer \(x_h \in D_h\coloneq X_h\cap D\) of the residual \(F - B(\xi_h)\)
in the norm of \(Y_h^*\),
\begin{equation}
  \tag{dPG}
  \label{eq:dpg}
  x_h \in \argmin_{\xi_h \in D_h} \Vert F - B(\xi_h) \Vert_{Y_h^*}.
\end{equation}
The existence of a solution to~\eqref{eq:dpg} requires further assumptions and follows in Proposition~\ref{prop:discrete_existence} for a model problem.

\subsection{Derivation of nonlinear dPG}
A formal Lagrange ansatz leads to the minimization of the Lagrange
functional \(\mathcal{L}: D_h \times Y_h \times \R \to \R\) defined for \((x_h, y_h, \lambda)\in X_h\times Y_h \times \R\) by
\[
  \mathcal{L}(x_h, y_h, \lambda)
  \coloneq
  F(y_h) - b(x_h; y_h)
  - \frac\lambda2 \big( a(y_h, y_h) - 1 \big).
\]
The stationary points \(x_h \in D_h\), \(y_h \in Y_h\), and
\(\lambda \in \R\) of \(\mathcal{L}\) are
characterized by the first derivatives of \(\mathcal{L}\) with respect to each
argument in the sense that, for all \(\eta_h \in Y_h\) and \(\xi_h \in X_h\),
\begin{align*}
  \lambda a(y_h, \eta_h) + b(x_h; \eta_h) = F(\eta_h),\quad b'(x_h; \xi_h, y_h)= 0,\quad a(y_h, y_h) = 1.
\end{align*}
For \(\eta_h = y_h\), this implies \(\lambda = F(y_h) - b(x_h; y_h)\). The
substitution of \(y_h\) by \(\lambda y_h\) leads to a modified system of
equations.
The resulting mixed formulation of the nonlinear dPG method seeks
\(x_h \in X_h\) and \(y_h \in Y_h\) with
\begin{equation*}
  \begin{aligned}
    a(y_h,\eta_h) + b(x_h; \eta_h)
    &= F(\eta_h)
    \quad\text{for all } \eta_h \in Y_h,
    \\
    b'(x_h; \xi_h, y_h)
    &= 0
    \quad\text{for all } \xi_h \in X_h.
  \end{aligned}
\end{equation*}
Notice that this is known for linear problems (there, \(b=b'(x_h;\pbullet)\)) \cite[Section~2.3]{MR2916380}.

\subsection{Equivalent mixed formulation}
It is known in linear problems that the dPG method is equivalent to the mixed
problem \eqref{eq:mixed} and this is generalized in this subsection to the
nonlinear problem \(B(x) = F\) at hand.
Any local (or global) minimizer of \(\Phi(\xi_h)\coloneq \Vert F - B\xi_h
\Vert_{Y_h^*}^2/2\) is a stationary point of \(\Phi\).
\begin{definition}[stationary point]
  Any \(x_h\in D_h\coloneq D\cap X_h\) is a stationary point of the dPG
  discretization~\eqref{eq:dpg} if any directional derivative of
  \(\Phi(\xi_h)\coloneq \Vert F - B\xi_h \Vert_{Y_h^*}^2/2\) vanishes at
  \(x_h\), i.e.,\ \(\lim_{\delta\rightarrow 0}
  (\Phi(x_h+\delta\xi_h)-\Phi(x_h))/\delta = 0\) for all \(\xi_h\in X_h\).
\end{definition}
Stationary points are exactly the solutions to \eqref{eq:mixed}.
\begin{theorem}[\eqref{eq:dpg} \(\Leftrightarrow\) \eqref{eq:mixed}]
  \label{thm:mixed_formulation}
  (a) Suppose \(x_h\) is a stationary point of~\eqref{eq:dpg} and \(y_h\)
  is the residual's Riesz representation
  (i.e.\ \(a(y_h, \pbullet) = F - b(x_h; \pbullet)\)) in \(Y_h\).
  Then \((x_h, y_h)\) solves~\eqref{eq:mixed}.

  (b) Suppose that \((x_h, y_h)\) solves~\eqref{eq:mixed}, then \(x_h\) is
  a stationary point of~\eqref{eq:dpg}.
\end{theorem}
\begin{proof}
  (a)
  For any \(\xi_h\in D_h\), the unique Riesz representation \(\varrho_h(\xi_h) \in Y_h\) of the residual
  \(F - b(\xi_h; \pbullet) \in Y_h^*\) satisfies
  \[
    \Phi(\xi_h) = \frac12 \Vert \varrho_h(\xi_h) \Vert_Y^2.
  \]
  Given the stationary point \(x_h\in D_h\) to~\eqref{eq:dpg} and \(\xi_h \in X_h\),
  consider \(\Phi(x_h + t \xi_h)\) as a scalar function of the real parameter
  \(t\) with a derivative zero at \(t = 0\).
  For \(\vert t\vert\) small such that \(x_h(t) \coloneq x_h + t \xi_h\in D_h\) and
  \(y_h(t) \coloneq \varrho_h(x_h(t))\), it follows
  \[
    a(y_h(t), \pbullet) + b(x_h(t); \pbullet) = F
    \quad\text{in } Y_h^*.
  \]
  A differentiation with respect to \(t\) shows for \(\dot{y}_h \coloneq
  \partial y_h(0) / \partial t\) and \(\dot{x}_h \coloneq \partial x_h(0) /
  \partial t = \xi_h\) that \(\dot{y}_h\) exists and is the Riesz
  representation of \(-b'(x_h; \xi_h, \pbullet) = a(\dot{y}_h, \pbullet)\)
  in \(Y_h\).
  Therefore, \(\Phi(x_h(t)) = a(y_h(t), y_h(t))/2\) is differentiable and
  the derivative vanishes at \(t = 0\), which leads to
  \[
    0 = a(\dot{y}_h, y_h)
    \quad\text{for } y_h \coloneq y_h(0).
  \]
  It follows that
  \[
    b'(x_h; \xi_h, y_h) = 0
    \quad\text{for all } \xi_h \in X_h.
  \]
  Since \(y_h = y_h(0) = \varrho_h(x_h)\), \((x_h, y_h)\)
  solves~\eqref{eq:mixed}.

  (b) Conversely, if \((x_h, y_h)\) solves~\eqref{eq:mixed} then,
  for any \(\xi_h \in D_h\) and the above notation for the Riesz
  representation \(y_h(t)\) of \(F - B x_h(t)\) in \(Y_h\),
  \[
    \Vert F - B x_h(t) \Vert_{Y_h^*}^2
    =
    a(y_h(t), y_h(t))
    =
    F(y_h(t)) - b(x_h(t); y_h(t))
  \]
  has a derivative with respect to \(t\) at \(t = 0\), namely, for \(y_h
  \coloneq y_h(0)\)
  \[
    2\,a(\dot{y}_h, y_h)
    =
    F(\dot{y}_h) - b'(x_h; \xi_h, y_h) - b(x_h; \dot{y}_h).
  \]
  Since \(b'(x_h; \xi_h, y_h) = 0\) and
  \(F(\dot{y}_h) - b(x_h; \dot{y}_h) = a(y_h, \dot{y}_h)\), this implies
  \(a(\dot{y}_h, y_h) = 0\).
  Recall \(\partial \Phi(x(t))/\partial t\vert_{t = 0} = a(\dot{y}_h, y_h)
  = 0\), and so \(x_h\) is a stationary point of \(\Phi\).
  \qed
\end{proof}

\begin{proposition}[necessary and sufficient second-order condition]
  \label{prop:higher_order}
  Assume that \(\Phi\) is twice differentiable.
  (a)\; If \(x_h\) solves~\eqref{eq:dpg}, then
  \begin{equation}
    \label{eq:bounded_second_derivative}
    b''(x_h; \xi_h, \xi_h, y_h)
    \leq
    \Vert b'(x_h; \xi_h, \pbullet) \Vert_{Y_h^*}^2 \quad\text{ for all }\xi_h \in X_h.
  \end{equation}
  (b)\; If, in addition,
  \begin{equation}
    \label{eq:strictly_bounded_second_derivative}
    b''(x_h; \xi_h, \xi_h, y_h)
    <
    \Vert b'(x_h; \xi_h, \pbullet) \Vert_{Y_h^*}^2 \quad\text{ for all }\xi_h \in X_h \setminus \{0\},
  \end{equation}
  then \(x_h\) is locally unique.
\end{proposition}
\begin{proof}
  The second derivative of
  \(\Phi(x_h(t))\) reads
  \(a(\partial^2y_h/\partial t^2, y_h)
    + \Vert \partial y_h/\partial t \Vert_Y^2\).
  Recall from the proof of Theorem~\ref{thm:mixed_formulation} for
  \(t=0\), that the Riesz representation \(\dot{y}_h = \partial
  y_h(0)/\partial t\) satisfies
  \[
    a(\dot{y}_h, \pbullet) = -b'(x_h; \xi_h, \pbullet)
    \quad\text{in } Y_h
    \quad\text{and}\quad
    \Vert \dot{y}_h \Vert_Y = \Vert b'(x_h; \xi_h, \pbullet)
    \Vert_{Y_h^*}.
  \]
  Another differentiation with respect to \(t\) shows that \(\ddot{y}_h
  \coloneq \partial^2 y_h(0)/\partial t^2\) satisfies
  \[
    a(\ddot{y}_h, \pbullet) = -b''(x_h; \xi_h, \xi_h, \pbullet)
    \quad\text{in } Y_h.
  \]
  Consequently, the second derivative of \(\Phi(x_h(t))\) at \(t=0\) is
  \begin{equation}
    \label{eq:second_derivative_phi}
    -b''(x_h; \xi_h, \xi_h, y_h) + \Vert b'(x_h; \xi_h, \pbullet)
    \Vert_{Y_h^*}^2.
  \end{equation}
  The assertion follows from this and standard arguments in the calculus of stationary and minimal points.
%
  \qed
\end{proof}

\begin{remark}[linear problems]
  For a linear problem, \(b''(x_h; \pbullet)\) vanishes
  and~\eqref{eq:strictly_bounded_second_derivative} holds.
  This implies local uniqueness in the linear situation (which is a global
  one).

  The uniqueness of the discrete solution is observed in
  numerical examples; cf. Theorem~\ref{thm:a_post_unique} for a sufficient condition in the model example below.
\end{remark}

\subsection{Abstract a~priori error analysis}
This section presents a best-approxima\-tion result based on a discrete inf-sup condition and the existence of a Fortin operator.

\begin{hypothesis}\label{hyp:fortin}
  Throughout this paper, assume that there exists a linear bounded
  projection \(\Pi_h: Y \to Y_h\) with
  \(\Pi_h\vert_{Y_h} = \id\vert_{Y_h}\) and
  \begin{equation}
    \label{eq:Fortin}
    b'(D_h; X_h, (1-\Pi_h)Y) = 0,
  \end{equation}
  i.e., for all \(x_h \in D_h\) and all \(y \in Y\), \(\Pi_h y \in Y_h\)
  satisfies \(b'(x_h; \xi_h, y - \Pi_h y) =0\) for all \(\xi_h \in X_h\).
  Let \(\Vert \Pi_h \Vert\) denote the bound of \(\Pi_h\) in \(L(Y;Y)\).
\end{hypothesis}

The following theorem generalizes \cite[Prop.~5.4.2]{MR3097958} to the nonlinear
problem at hand. A sufficiently fine initial triangulation guarantees that \(B(x, \varepsilon)\cap X_h\subset D_h\) is nonempty.
\begin{theorem}[discrete inf-sup condition]
  \label{thm:discrete_inf_sup}
  Given a regular solution \(x\) to \(B(x) =F\), there exists
  an open ball \(B(x, \varepsilon)\coloneq \{\tilde x\in X\,|\, \Vert x-\tilde x\Vert_X<\varepsilon\}\) of radius \(\varepsilon > 0\) around
  \(x\) such that, for all \(\tilde x_h \in B(x, \varepsilon)\cap X_h\subset D_h\), the
  following discrete inf-sup condition holds
  \[
    0 < \frac{\beta(x;X_h,Y_h)}{2\Vert \Pi_h\Vert}
    \leq \beta(\tilde x_h; X_h, Y_h)
    \coloneq
    \inf_{\xi_h \in \sphere{X_h}} \sup_{\eta_h \in \sphere{Y_h}}
    b'(\tilde x_h; \xi_h, \eta_h).
  \]
\end{theorem}
\begin{proof}
  The continuous inf-sup condition~\eqref{eq:continuous_inf_sup} and the continuity of \(\Grad B\) in \(D\) lead to some \(\varepsilon\) such that
  \begin{align}
    B(x,\varepsilon)&\subset D,\label{eq:cond_eps_d}\\
    \beta(x) / 2 &\leq \inf_{\xi\in B(x,\varepsilon)} \beta(\xi).\label{eq:cond_eps_beta}
  \end{align}
  Then \(\tilde x_h \in B(x, \varepsilon)\cap X_h\) and~\eqref{eq:Fortin} imply
  \begin{align*}
    \beta(x) / 2
    &\leq
    \beta(\tilde x_h)
    \leq
    \inf_{\xi_h \in \sphere{X_h}} \sup_{\eta \in \sphere{Y}}
    b'(\tilde x_h; \xi_h, \eta)\\
    &=
    \inf_{\xi_h \in \sphere{X_h}} \sup_{\eta \in \sphere{Y}}
    b'(\tilde x_h; \xi_h, \Pi_h\eta)\\
    &=
    \inf_{\xi_h \in \sphere{X_h}} \sup_{\eta \in \sphere{Y}}
    \Vert \Pi_h \eta \Vert_Y b'\big( \tilde x_h; \xi_h, \Pi_h\eta \big/
    \Vert \Pi_h\eta \Vert_Y \big)\\
    &\leq
    \Vert \Pi_h \Vert
    \inf_{\xi_h \in \sphere{X_h}} \sup_{\eta \in \sphere{Y}}
    b'\big( \tilde x_h; \xi_h, \Pi_h\eta \big/
    \Vert \Pi_h\eta \Vert_Y \big)\\
    &=
    \Vert \Pi_h \Vert
    \inf_{\xi_h \in \sphere{X_h}} \sup_{\eta_h \in \sphere{Y_h}}
    b'(\tilde x_h; \xi_h, \eta_h).
  \end{align*}
  Hence, any \(\tilde x_h \in B(x, \varepsilon)\) satisfies \(0 < \frac{\beta(x)}{2\Vert \Pi_h \Vert}
    \leq
    \beta(\tilde x_h; X_h, Y_h)\).
  \qed
\end{proof}

\begin{remark}[converse of Theorem~\ref{thm:discrete_inf_sup}]
  Given the discrete inf-sup condition
  \begin{equation}
  \label{eq:disc_inf_sup}
    0 < \inf_{\xi_h \in \sphere{X_h}} \sup_{\eta_h \in \sphere{Y_h}}
    b'(\tilde x_h; \xi_h, \eta_h)
  \end{equation}
  \changed{at some point \(\widetilde x_h \in D_h\)}, the techniques of
  \cite[Lemma~10]{MR3576569} guarantee the existence of a
  linear bounded projection \(\Pi_h(\widetilde x_h):Y\to Y_h\)
  with~\eqref{eq:Fortin}, which depends on \(\widetilde x_h\).
  The above proof shows that \changed{the existence of \(\Pi_h(\widetilde
  x_h)\)} is also sufficient for \eqref{eq:disc_inf_sup}.
  The class of model examples allows for the simple more uniform
  Hypothesis~\ref{hyp:fortin} with \(\Pi_h(\widetilde x_h)=\Pi_h\)
  independent of \(\widetilde x_h\in D_h\).
\end{remark}

\begin{theorem}[local best-approximation]
  \label{thm:ba}
  Given a regular solution \(x\) to \(B(x)=F\), there exist positive constants \(\varepsilon > 0\) and \(C(x,\varepsilon)>0\) such
  that any solution \((x_h, y_h)\) to~\eqref{eq:mixed} with
  \(\Vert x - x_h \Vert_X < \varepsilon\) satisfies
  \[
    \Vert x-x_h \Vert_X + \Vert y_h\Vert_Y
    \leq C(x,\varepsilon) \inf_{\xi_h \in X_h} \Vert x-\xi_h \Vert_X.
  \]
\end{theorem}
The proof of the theorem requires the following lemma.
\begin{lemma}\label{lem:taylor_b}
 Any \(\varepsilon>0\) and \(x_h\in B(x,\varepsilon)\subset D\) satisfy
 \begin{align}
  \Vert b'(x_h; x - x_h, \pbullet) &- b(x; \pbullet)+ b(x_h; \pbullet) \Vert_{Y^*}\notag\\
  &\leq 2 \sup_{\xi\in B(x,\varepsilon)} \Vert \Grad B(x)- \Grad B(\xi)\Vert_{L(X;Y^*)} \Vert x - x_h \Vert_X\label{eq:taylor_b_1}\\
  \Vert B(x) - B(x_h)\Vert_{Y^*}&\leq \sup_{\xi\in B(x,\varepsilon)} \Vert \Grad B(\xi)\Vert_{L(X;Y^*)} \Vert x - x_h \Vert_X.\label{eq:taylor_b_2}
 \end{align}
\end{lemma}
\begin{proof}
 Given any \(\eta\in\sphere{Y}\), the Taylor's formula of \(b\) at \(x_h\) with remainder reads
  \begin{align*}
    b'(x_h; x - x_h, \eta) &- b(x; \eta)+ b(x_h; \eta) \\
    &= \int_0^1 \Big( b'(x_h;x-x_h,\eta) - b'(x_h+s(x-x_h);x-x_h,\eta) \Big) \ds.
  \end{align*}
  Since \(\Vert x - x_h\Vert_X <\varepsilon\) implies \(\Vert x - (x_h+s(x-x_h))\Vert_X <\varepsilon\) for \(0\leq s \leq 1\), the triangle inequality proves
  \begin{align*}
    b'(x_h; x - x_h, \eta) &+ b(x_h; \eta)
    - b(x; \eta)\\
    &\leq 2\sup_{\xi\in B(x,\varepsilon)} \vert b'( x;x-x_h,\eta) - b'(\xi;x-x_h,\eta) \vert\\
    &\leq 2 \sup_{\xi\in B(x,\varepsilon)} \Vert \Grad B(x)- \Grad B(\xi)\Vert_{L(X;Y^*)} \Vert x - x_h \Vert_X.
  \end{align*}
  Since \(\eta\in\sphere{Y}\) is arbitrary, this implies \eqref{eq:taylor_b_1}.
  The assertion \eqref{eq:taylor_b_2} follows from the same arguments without the term \(b'(x_h; x - x_h, \eta)\).\qed
\end{proof}

\begin{proof}[of \changed{Theorem~\ref{thm:ba}}]
%
  Let \(\tilde x_h\) be the best-approximation to \(x\) in \(X_h\), i.e.,
  \[
    \Vert x - \tilde x_h \Vert_X
    = \inf_{\xi_h \in D_h} \Vert x - \xi_h \Vert_X \leq \Vert x - x_h \Vert_X < \varepsilon.
  \]
  Suppose \(\varepsilon>0\) satisfies \eqref{eq:cond_eps_d}-\eqref{eq:cond_eps_beta} and, with the continuity of \(\Grad B\) at \(x\),
  \begin{equation}
   \label{eq:cond_eps_bound_db}
    \sup_{\xi\in B(x,\varepsilon)} \Vert \Grad B(\xi)\Vert_{L(X;Y^*)} \leq 2 \Vert \Grad B(x)\Vert_{L(X;Y^*)}.
  \end{equation}
  The discrete inf-sup condition from
  Theorem~\ref{thm:discrete_inf_sup} plus the Brezzi splitting lemma
  \cite[Thm.~4.3 in Ch.~III]{MR2322235} with inf-sup constants,
  \(\beta(x)/2\) and \(1\), and continuity constants, \(2 \Vert \Grad
  B(x)\Vert_{L(X;Y^*)}\) and \(1\), for the
  bilinear form \(b'(x_h; \pbullet, \pbullet)\) and scalar product
  \(a\) prove the global inf-sup condition \(0<\gamma\leq\beta(x_h;X_h,Y_h)\) for
  \[
    \gamma
    \coloneq
    \hspace{-2mm}\inf_{(\tilde\xi_h, \tilde\eta_h) \in \sphere{X_h \times Y_h}}
    \sup_{(\xi_h, \eta_h) \in \sphere{X_h \times Y_h}}
    \hspace{-1mm}\Big(b'(x_h; \tilde\xi_h, \eta_h)
    + b'(x_h; \xi_h , \tilde\eta_h)
    + a(\tilde\eta_h,\eta_h)\Big).
  \]
  independent of \(\varepsilon\) with \eqref{eq:cond_eps_d}--\eqref{eq:cond_eps_beta} and \eqref{eq:cond_eps_bound_db}.
  Given \(\gamma>0\) and \(\beta(x)>0\) suppose, for some smaller
  \(\varepsilon>0\) if necessary, that \(\varepsilon>0\) satisfies
  \eqref{eq:cond_eps_d}--\eqref{eq:cond_eps_beta},
  \eqref{eq:cond_eps_bound_db}, and, from the continuity of \(\Grad B\) at
  \(x\),
  \begin{equation}
   \label{eq:cond_eps_cont_db}
    \sup_{\xi\in B(x,\varepsilon)} \Vert \Grad B(x)- \Grad B(\xi)\Vert_{L(X;Y^*)} \leq \min\{\gamma/4,\beta(x)/8\}.
  \end{equation}
  For the best-approximation \(\tilde y_h=0\) to \(y=0\) in \(Y_h\) and
  \((\tilde\xi_h, \tilde\eta_h) = (\tilde x_h - x_h, \tilde y_h - y_h)\),
  this implies the existence of
  \((\xi_h, \eta_h) \in \sphere{X_h \times Y_h}\) with
  \[
    \gamma \big( \Vert \tilde x_h - x_h \Vert_X
    + \Vert y_h \Vert_Y \big)
    \leq
    b'(x_h; \tilde x_h - x_h, \eta_h)
    - b'(x_h; \xi_h, y_h)
    - a(y_h, \eta_h).
  \]
  Since \((x_h, y_h)\) solves~\eqref{eq:mixed} and \(\tilde y_h = 0\), this
  leads to
  \[
    \gamma \big( \Vert \tilde x_h - x_h \Vert_X
    + \Vert y_h \Vert_Y \big)
    \leq
    b'(x_h; \tilde x_h - x, \eta_h)
    + b'(x_h; x - x_h, \eta_h)
    + b(x_h; \eta_h) - b(x; \eta_h).
  \]
  Lemma~\ref{lem:taylor_b} and \eqref{eq:cond_eps_cont_db} imply
  \begin{align*}
    b'(x_h; x - x_h, \eta_h) &+ b(x_h; \eta_h)
    - b(x; \eta_h)\leq \gamma\Vert x - x_h \Vert_X/2.
  \end{align*}
  The combination of the preceding two displayed formulae reads
  \begin{align*}
   \frac\gamma2 \Vert \tilde x_h - x_h \Vert_X + \gamma  \Vert y_h \Vert_Y &\leq b'(x_h;\tilde x_h - x_h,\eta_h).
  \end{align*}
  With \eqref{eq:cond_eps_bound_db}, this is bounded from above by
  \begin{align*}
   \Vert \Grad B (x_h)\Vert_{L(X;Y^*)} \Vert x-\tilde x_h\Vert_X \leq 2 \Vert \Grad B (x_h)\Vert_{L(X;Y^*)} \Vert x-\tilde x_h\Vert_X.
  \end{align*}
  The triangle inequality concludes the proof.
  \qed
\end{proof}

\begin{remark}
  \changed{Under further smoothness conditions of the nonlinear mapping
    \(b'\) the local existence and uniqueness of a discrete solution, e.g.,
    follows from \cite[Thm.~2]{MR1310318}.}
\end{remark}

\begin{remark}
  \label{rem:kantorovich}
  The Newton-Kantorovich theorem \cite[Section~5.2]{MR816732} is
  \changed{another} tool
  for the proof of the existence of discrete solutions close to the regular
  solution. In the model problem of Section~\ref{sec:model_problem}, the
  higher Fr\'echet derivatives for this argument do \emph{not} exist, cf.
  Remark~\ref{rem:second_derivative} for details.
\end{remark}

\subsection{Abstract a~posteriori error analysis}
This subsection is devoted to a brief abstract a~posteriori error analysis
of the nonlinear dPG. Given a discrete approximation \(x_h\) close to the
regular solution \(x\) to \(B(x)=F\), the residual \(F-B(x_h)\in Y^*\) has
a norm \(\Vert F-B(x_h)\Vert_{Y^*}\) that, in principle, is accessible in
the sense that lower and upper bounds may be computable. The latter issue
is a typical general task in the a ~posteriori error analysis and will be
adressed in Section~\ref{sec:model_problem} for a model example.

\begin{theorem}[local a~posteriori analysis]
\label{thm:local_a_post}
  Let \(x\) be a regular solution to \(B(x)=F\) with inf-sup constant
  \(\beta(x)\) from \eqref{eq:continuous_inf_sup}. Then there exists some
  \(\varepsilon>0\) such that any \(x_h\in B(x,\varepsilon)\subset D\)
  satisfies
  \[
    \frac{\beta(x)}{4} \Vert x - x_h\Vert_X \leq \Vert F-B(x_h)\Vert_{Y^*}
    \leq 2 \Vert\Grad B(x)\Vert_{L(X;Y^*)} \Vert x - x_h\Vert_X.
  \]
\end{theorem}
\begin{proof}
  With the choice of \(\varepsilon>0\) from the proof of
  Theorem~\ref{thm:ba} it follows
  \eqref{eq:cond_eps_d}-\eqref{eq:cond_eps_beta} and
  \eqref{eq:cond_eps_bound_db}-\eqref{eq:cond_eps_cont_db}.
  The \changed{continuous} inf-sup condition \eqref{eq:continuous_inf_sup} implies
  the existence of \(\eta\in \sphere{Y}\) with
  \begin{align*}
    \frac{\beta(x)}{2} \Vert x - x_h\Vert_X &\leq b'(x_h;x-x_h,\eta)\\
    &\leq b(x;\eta) - b(x_h;\eta) + \vert b'(x_h;x-x_h,\eta) - b(x;\eta) + b(x_h;\eta)\vert.
  \end{align*}
  Lemma~\ref{lem:taylor_b} for the last term, \(b(x;\eta)=F(\eta)\), and
  \eqref{eq:cond_eps_cont_db} show
  \[
    \frac{\beta(x)}{2} \Vert x - x_h\Vert_X \leq F(\eta) - b(x_h;\eta) +
    \frac{\beta(x)}{4} \Vert x - x_h\Vert_X.
  \]
  This proves the asserted reliability
  \[
    \frac{\beta(x)}{4} \Vert x - x_h\Vert_X \leq \Vert F(\eta) -
    b(x_h;\eta)\Vert_{Y^*}.
  \]
  To prove the efficiency, utilize \(F=B(x)\), Lemma~\ref{lem:taylor_b},
  and \eqref{eq:cond_eps_bound_db} to verify
  \begin{align*}
    \Vert F-B(x_h)\Vert_{Y^*} &= \Vert B(x) - B(x_h)\Vert_{Y^*}\\ &\leq
    2\Vert \Grad B(x)\Vert_{L(X;Y^*)} \Vert x-x_h\Vert_X. \quad\qed
  \end{align*}
\end{proof}

\begin{remark}
  Since \(y=0\) and \(y_h\) is computed, the a~posteriori error \(\Vert y -
  y_h\Vert_Y = \Vert y_h\Vert_Y\) is already an error estimator and can be
  added on both sides of the reliability (resp. efficiency) a~posteriori
  error estimate. This justifies the usage of the extended residual \(\Vert
  F-a(y_h,\pbullet)-b(x_h;\pbullet)\Vert_{Y^*} + \Vert y_h\Vert_Y\) of the
  system \eqref{eq:mixed}.
\end{remark}
\begin{remark}
  The constants \(\beta(x)/4\) (resp. \(2\Vert\Grad B(x)\Vert_{L(X;Y^*)}\))
  in Theorem~\ref{thm:local_a_post} follow from the choice of
  \(\varepsilon\) in the a~priori error analysis in the proof of
  Theorem~\ref{thm:ba}. For smaller and smaller values of \(\varepsilon\),
  those constants could be replaced by any number \(<\beta(x)\) (resp.
  \(>\Vert\Grad B(x)\Vert_{L(X;Y^*)}\)) in the following sense. For any
  \(0<\lambda<1\) there exists some \(\varepsilon>0\) such that any \(x_h\in
  B(x,\varepsilon)\) satisfies \(\lambda\beta(x)\leq \Vert
  F-B(x_h)\Vert_{Y^*}\leq (1+\lambda) \Vert\Grad B(x)\Vert_{L(X;Y^*)}\).
\end{remark}

\section{Model problem}
\label{sec:model_problem}
This section introduces a nonlinear model problem and a low-order dPG
discretization and establishes two further equivalent characterizations of
the nonlinear dPG method: reduced discretization and weighted
least-squares.
\subsection{Convex energy minimization}
\label{sec:convex_energy_minimization}
The nonlinear model problem involves a nonlinear function
\(\phi \in C^2(0, \infty)\) with \(0 < \gamma_1 \leq \phi(t) \leq \gamma_2\) and
\(0 < \gamma_1 \leq \phi(t) + t \phi'(t) \leq \gamma_2\) for all \(t \geq 0\) and
universal positive constants \(\gamma_1, \gamma_2\).
%
Given \(f \in L^2(\Omega)\) and the convex function \(\varphi\), \(\varphi(t) \coloneq \int_0^t s \, \phi(s) \ds\) for \(t\geq 0\), the model problem minimizes the
energy functional
\[
  E(v)
  \coloneq \int_\Omega \varphi(\vert \grad v(x)\vert) \dx - \int_\Omega fv \dx
  \quad\text{among all }
  v \in  \changed{H^1_0(\Omega)}.
\]
The convexity of \(\varphi\) and the above assumptions on \(\phi\) lead to growth-conditions and sequential weak
lower semicontinuity of \(E\) and guarantee the unique existence of a
minimizer \(u\) of \(E\) in \changed{\(H^1_0(\Omega)\)}
\cite[Thm.~25.D]{MR1033498}.
The equivalent Euler-Lagrange equation reads
\begin{equation}
  \label{eq:EulerLagrange}
  \int_\Omega \phi(\vert \grad u \vert) \grad u \cdot \grad v \dx
  =
  \int_\Omega fv \dx
  \quad\text{for all } v \in \changed{H^1_0(\Omega)}
\end{equation}
and has the unique solution \(u\) in \changed{\(H^1_0(\Omega)\)}.
The stress variable \(\sigma(A) \coloneq \phi(\vert A \vert) A\)
defines a function \(\sigma \in C^1(\R^n; \R^n)\) with Fr\'echet derivative
\begin{equation}\label{eq:dsigma}
  \Grad \sigma(A)
  =
  \phi(\vert A \vert) I_{n\times n}
  + \phi'(\vert A \vert) \vert A \vert \sign(A) \otimes \sign(A)
\end{equation}
with the sign function \(\sign(A) \coloneq A\,/\,\vert A \vert\) for
\(A \in \R^n \setminus \lbrace 0 \rbrace\) and the closed unit ball \(\sign(0) \coloneq
\overline{B(0, 1)}\) in \(\R^n\).
The prefactor \(\phi'(\vert A \vert) \vert A \vert\) makes \(\Grad\sigma\) a
continuous function in \(\R^n\).
In fact \(\Grad\sigma\in C^0(\R^{n\times n}_{\operatorname{sym}})\) is
bounded with eigenvalues in the compact interval
\([\gamma_1,\gamma_2]\subset (0,\infty)\).
\begin{remark}[\(\operatorname{Lip}(\sigma) \leq \gamma_2\)]
\label{rem:lipschitz}
 For \(A,B\in \R^n\), the argument
  \(
    \sigma(A) - \sigma(B)
    =
    \int_0^1 \Grad \sigma(sA+(1-s)B) (A-B) \ds
  \)
  and~\eqref{eq:dsigma} imply the global Lipschitz continuity of \(\sigma\),
  \[
    \vert \sigma(A) - \sigma(B) \vert
    \leq
    \int_0^1 \vert \Grad \sigma(sA+(1-s)B) (A-B) \vert \ds
    \leq
    \gamma_2 \vert A - B \vert.
  \]
\end{remark}
\begin{example}
  \label{ex:density}
  In the following examples, \(0 \leq \phi'' \leq 2\) is bounded as well as
  \(\phi'\) and \(\Grad\sigma\) from \eqref{eq:dsigma} is globally Lipschitz continuous.
  (a)\;
  \(\phi(t) \coloneq 2 + (1 + t)^{-2}\)
  with \(\gamma_1 = 1 < \gamma_2 = 3\) \cite{MR1364401} and \(\operatorname{Lip}(\Grad\sigma) \leq 4\) and
  (b)\;
  \(\phi(t) \coloneq 2 - (1 + t^2)^{-1}\)
  with \(\gamma_1 = 1 < \gamma_2 = 4\) and \(\operatorname{Lip}(\Grad\sigma) \leq 2\).
\end{example}

\begin{remark}[second derivative]
\label{rem:second_derivative}
 A formal calculation with \(s(j) \coloneq (\sign A)_j\),
\(s(j,k) \coloneq (\sign A)_j (\sign A)_k\) etc.\ and the Kronecker symbol
\(\delta_{jk}\) for \(j, k, \ell = 1, \dots, n\) leads at any \(A \in \R^n\) to
\[
  \Hess \sigma(A)_{j,k,\ell}
  =
  \phi'(\vert A \vert)(\delta_{jk} s(\ell) + \delta_{j\ell} s(k)
  + \delta_{k\ell}s(j))
  + (\phi''(\vert A \vert) \vert A \vert - \phi'(\vert A \vert)) s(j,k,\ell).
\]
Although \(\Hess \sigma(A)\) may be bounded (at least in the
Example~\ref{ex:density}.a and b), it may be discontinuous for
\(A \to 0\).
In Example~\ref{ex:density}.b, \(\phi'(0) = 0\) and
\(\Hess \sigma\) is continuous with \(\Hess \sigma(0) = 0\). The associated
trilinear form \(b''(x;\pbullet)\), however, is not well-defined on
\(X\times Y\times Y\) because the product of three Lebesgue functions in \(L^2(\Omega)\) is, in general, not in \(L^1(\Omega)\).
\end{remark}

\subsection{Breaking the test spaces}
\label{sec:breaking}
Let $\Dom \subseteq \R^n$ be a bounded Lipschitz domain with polyhedral
boundary $\partial\Dom$.
Let $\Tri$ denote a regular triangulation of the domain $\Dom$ into
$n$-simplices and let $\Edges$ (resp.\ $\Edges(T)$) denote the set of all
sides in the triangulation (resp.\ of an $n$-simplex $T\in\Tri$).

The unit normal vector $\nu_T$ along the boundary $\partial T$ of an
$n$-simplex $T\in\Tri$ (is constant along each side of \(T\) and) points
outwards.
For any side $E=\partial T_+\cap\partial T_-\in \Edges$ shared by
two simplices, the enumeration of the neighbouring simplices $T_\pm$ is
globally fixed and so defines a unique orientation of the unit normal
\(\nu_E = \nu_{T_+}\vert_E\).
Let $h_T$ denote the diameter of $T\in\Tri$,
$h_{\textup{max}}\coloneq\max_{T\in\Tri} h_T\leq\operatorname{diam}(\Dom)$
and $h_\Tri|_K=h_K$ for any $K\in\Tri$.
The barycenter $\Mid(T)$ of $T\in \Tri$ defines the piecewise constant
function $\Mid(\Tri)\in P_0(\Tri;\R^n)$ by $\Mid(\Tri)|_K \coloneq \Mid(K)$
for any $K\in\Tri$ and $\Mid(E)$ is the barycenter of $E\in \Edges$.
The piecewise affine function \(\pbullet - \Mid(\Tri) \in P_1(\Tri;\R^n)\) equals \(x -
\Mid(T)\) at \(x \in T \in \Tri\).

Recall that \(H^k(\Tri)
\coloneq \prod_{T \in \Tri} H^k(T) \coloneq \{v \in L^2(\Omega) \,|\,
\forall T \in \Tri,\, v\vert_T \in H^k(T)\}\) denotes the piecewise Sobolev
space.
Define the discrete spaces
\begin{align*}
  P_k(T)&\coloneq \{  v_k\in L^\infty (T)\,\mid \;  v_k \text{ is polynomial on } T \text{ of degree }\leq k\},\\
  P_k(\Tri)&\coloneq \{  v_k\in L^\infty (\Dom)\,\mid \;  \forall T\in\Tri, \, v_k|_T\in P_k(T)\},\\
  P_k(\Tri;\R^n)&\equiv P_k(\Tri)^n,\\
 S^k_0(\Tri) &\coloneq P_k(\Tri)\cap H^1_0(\Dom),\\
 RT_k(\Tri) &\coloneq  \{ q_k \in H(\Div,\Dom)  \,|\, \exists A \in P_k(\Tri;\R^n),\exists b\in P_k(\Tri),\\
   &\quad\quad\quad\quad\quad\quad\quad\quad\quad\quad q_k=A + b(\pbullet - \Mid(\Tri))\},\\
 CR^1(\Tri)
 &\coloneq  \{ v_\CR \in  P_1(\Tri)\,|\, \forall E\in\Edges(\Omega),\,
 v_\CR\text{ continuous at} \Mid(E)\},\\
 CR^1_0(\Tri)
 &\coloneq  \{ v_\CR \in CR^1(\Tri) \,|\, \forall E\in\Edges(\partial\Dom),\, v_\CR(\Mid(E))=0\},\\
 P_k(\Edges)&\coloneq \{t_k \in L^2(\partial\Tri)\,|\, t_k|_E \in P_k(E) \text{ for any }E\in\Edges\}.
\end{align*}

\begin{definition}
 For a triangulation $\Tri$ with skeleton $\partial\Tri\coloneq \bigcup_{T\in\Tri} \bigcup_{E\in\Edges(T)} E$ and $T\in\Tri$, recall the local trace spaces $H^{1/2}(\partial T)$ and $H^{-1/2}(\partial T) = (H^{1/2}(\partial T))^\star$ and
\begin{align*}
  H^{-1/2}(\partial\Tri)
  &\coloneq
  \{ t = (t_T)_{T\in\Tri}\in  {\textstyle \prod_{T\in\Tri} }
    H^{-1/2}(\partial T) \,|\,\\
    &\hspace{5em}
    \exists q\in H(\Div,\Dom),\forall T \in \Tri,t_T=(q|_T)|_{\partial
    T}\cdot \nu_T\}
 \end{align*}
endowed with the minimal extension norm, for $t\in H^{-1/2}(\partial\Tri)$,
\begin{align*}
 \Vert t \Vert_{H^{-1/2}(\partial\Tri)}&\coloneq \min \{\NormHDivDom{q}\,|\, q\in H(\Div,\Dom),\forall T \in
\Tri, t_T=(q|_T)|_{\partial T}\cdot \nu_T\}.
\end{align*}
The duality brackets $\langle{\pbullet},{\pbullet}\rangle_{\partial T}$ in $ H^{-1/2}(\partial T) \times  H^{1/2}(\partial T)$ extend the $L^2$ scalar product in $L^2(\partial T)$ and lead to the duality bracket on the skeleton for any $t=(t_T)_{T\in\Tri}\in\prod_{T\in\Tri} H^{-1/2}(\partial T)$ and $s=(s_T)_{T\in\Tri}\in\prod_{T\in\Tri}H^{1/2}(\partial T)$ defined by
\begin{align*}
 \langle t, s\rangle_{\partial\Tri} \coloneq  \sum_{T\in\Tri} \langle{t_T},{s_T}\rangle_{\partial T}.
\end{align*}
\end{definition}

\begin{remark}[\(RT_0(\Tri)\equiv P_0(\Edges)\)]
\label{rem:identification_RT_P0Edges}
  The spaces $RT_0(\Tri)$ and $P_0(\Edges)$ are isomorphic \cite[Lemma
  3.2]{MR3279489} in the sense that any $q_\RT\in RT_0(\Tri)$ and
  $E\in\Edges$ with fixed unit normal vector $\nu_E$ satisfies
  $q_\RT|_E\cdot \nu_E \in P_0(E)$.
  Conversely, for any $t_0\in P_0(\Edges)$, there exists a unique $q_\RT\in
  RT_0(\Tri)$ with $q_\RT|_E\cdot \nu_E =t_0|_E$ for any $E\in\Edges$, in
  short notation \(q_\RT\cdot \nu=t_0\) in \(\partial\Tri\).
  Since $\Vert t_0 \Vert_{H^{-1/2}(\partial\Tri)}\approx
  \NormHDivDom{q_\RT}$, this identification justifies the embedding
  $P_0(\Edges)\subseteq H^{-1/2}(\partial\Tri)$, where any $T\in\Tri$ and
  $E\in\Edges(T)$ satisfy $(q_\RT\cdot \nu_T)|_E= \pm t_0|_E$ with the
  sign \(\pm = \nu_T \cdot \nu_E\) depending on the (globally fixed) choice of the
  orientation of the unit normal \(\nu_E \in \{\nu_{T_\pm}|_E\}\).
\end{remark}

\begin{definition}
 Define $S_0\in P_0(\Tri;\R^{n\times n})$ and $H_0:L^2(\Omega)\rightarrow P_0(\Tri;\R^n)$ for \(T\in\Tri\) and \(f\in\Ltwo\) by
\begin{equation}
  \label{eq:S0H0}
    \begin{split}
      S_0|_T&\coloneq \Pi_0 ((\pbullet-\Mid(T))\otimes (\pbullet-\Mid(T)) ),\\
      H_0 f&\coloneq \Pi_0(f(\pbullet - \Mid(\Tri)))\in P_0(\Tri;\R^n).
    \end{split}
\end{equation}
\end{definition}

\begin{remark}\label{rem:properties_S0}
 An analysis of the eigenvalues of the piecewise symmetric positive semi-definite matrix \(S_0\) shows that any \(T\in\Tri\) and $v\in\R^n$ satisfies
 \[
  \vert v \vert \leq \vert (I_{n\times n} + S_0|_T) v \vert \leq (1+h_T^2) \vert v \vert
  \text{ and }
  \vert v \vert \leq \vert (I_{n\times n} + S_0|_T)^{1/2} v \vert \leq (1+h_T) \vert v \vert.
 \]
 Furthermore, \(\NormLzDom{H_0 f}\leq h_{\textup{max}} \NormLzDom{(1-\Pi_0)
 f}\) for the maximal mesh-size \(h_{\textup{max}} = \max h_\Tri\) in
 \(\Tri\).
\end{remark}

\subsection{Lowest-order dPG discretization}\label{sec:dpg_discretization}
The nonlinear model problem of this paper concerns the nonlinear map
\(\sigma: \R^n \to \R^n\) of Subsection~\ref{sec:convex_energy_minimization}. A piecewise integration by parts in \eqref{eq:EulerLagrange} and the introduction of the new variable \(t\coloneq \sigma(\nabla u)\cdot \nu\) on \(\partial\Tri\) leads to
the nonlinear primal dPG method with \(F(v)\coloneq \int_\Omega f v\dx\) and
\(b: X \times Y \to \R\) for \(X \coloneq \changed{H^1_0(\Omega)} \times H^{-1/2}(\partial\Tri)\)
and \(Y \coloneq H^1(\Tri)\) defined by
\begin{equation}
\label{eq:def_b}
  b(u, t; v) \coloneq
  \int_\Omega \sigma(\grad u) \cdot \gradNC v \dx
  - \langle t, v \rangle_{\partial\Tri}
  \eqcolon \langle B(u,t), y \rangle_Y.
\end{equation}
for all \(x = (u, t) \in X \coloneq \changed{H^1_0(\Omega)} \times
H^{-1/2}(\partial\Tri)\) and \(y = v \in Y = H^1(\Tri)\) with associated
norms and the scalar product \(a\) in \(Y\).
Given the subspaces \(X_h \coloneq S^1_0(\Tri) \times P_0(\Edges)\)
and \(Y_h \coloneq P_1(\Tri)\), the discrete problem minimizes the
residual norm and seeks \((u_h, t_h) = x_h \in X_h\) with
\begin{equation}
  \label{eq:model_dpg}
  \Vert F -B(x_h)\Vert_{Y_h^*}
  = \min_{\xi_h \in X_h} \Vert F- B(\xi_h)\Vert_{Y_h^*}.
\end{equation}
The derivative \(\Grad \sigma : \R^n \to \R^{n \times n}\) gives rise to
the map
\begin{equation}
\label{eq:model_derivative_b}
  b'(u,t; w, s, v)
  \coloneq
  \int_\Omega \grad w \cdot (\Grad \sigma(\grad u) \gradNC v) \dx
  - \langle s, v \rangle_{\partial \Tri}.
\end{equation}
This defines a bounded bilinear form \(b'(u,t;\pbullet) : X\times
Y \to \R\) for any \(x = (u,t) \in X\) and the operator \(B\) associated with \(b\) belongs to \(C^1(X;Y^*)\).
Recall the equivalent mixed formulation from~\eqref{eq:mixed} for the model
problem at hand, which seeks \((u_h, t_h) \in X_h\) and \(v_h \in Y_h\) with
\begin{equation}
  \label{eq:model_mixed}
  \begin{split}
    a(v_h, \eta_h) + b(u_h, t_h; \eta_h)
    &= F(\eta_h)
    \quad\text{for all } \eta_h \in Y_h,\\
    b'(u_h, t_h; w_h, s_h, v_h)
    &=0
    \quad\text{for all } (w_h, s_h) \in X_h.
  \end{split}
\end{equation}

\begin{remark}[regular solution]
  Since \(\Grad \sigma(\nabla u)\in L^\infty(\Dom;\R^{n\times
  n}_{\operatorname{sym}})\) uniformly \changed{positive} definite, the splitting
  lemma from the linear theory \cite[Thm.~3.3]{MR3521055} implies the
  inf-sup condition \eqref{eq:continuous_inf_sup} for the nondegenerate
  bilinear form \(b'(x;\pbullet,\pbullet):X\times Y\rightarrow \R\). Hence,
  the solution \(x\in X\) to \(B(x)=F\) is regular.
\end{remark}

\subsection{Reduced discretization}
The dPG discretization~\eqref{eq:model_dpg} can be simplified to a
modified problem that seeks \((u_h, v_h) \in S^1_0(\Tri) \times CR^1_0(\Tri)\) with
 \begin{equation}
    \label{eq:model_reduced}
    \tag{R}
    \begin{split}
      a(v_h, w_\CR) + \int_\Omega \sigma(\grad u_h) \cdot \gradNC w_\CR \dx
      &=
      \int_\Omega f w_\CR \dx
      \text{ for all } w_\CR \in CR^1_0(\Tri),\\
      \int_\Omega \grad w_\Conf
      \cdot \big( \Grad \sigma(\grad u_h) \gradNC v_h \big) \dx
      &= 0
      \text{ for all } w_\Conf \in S^1_0(\Tri).
    \end{split}
  \end{equation}
\begin{theorem}[\eqref{eq:model_mixed}\(\iff\)\eqref{eq:model_reduced}]
  \label{thm:reduced_formulation}
  (a)\;
  If \((u_h, t_h; v_h) \in X_h \times Y_h\) solves~\eqref{eq:model_mixed},
  then \(v_h \in CR^1_0(\Tri)\) and \((u_h, v_h) \in S^1_0(\Tri) \times
  CR^1_0(\Tri)\) solves~\eqref{eq:model_reduced}.

  (b)\;
  \changed{For any solution \((u_h, v_h) \in S^1_0(\Tri) \times CR^1_0(\Tri)\)}
  to~\eqref{eq:model_reduced}, there exists a unique \(t_h\in P_0(\Edges)\)
  such that \((u_h,t_h; v_h)\) solves~\eqref{eq:model_mixed}.
\end{theorem}
The proof utilizes the following discrete inf-sup condition of a {\em linear}
primal dPG method \cite{MR3093480}. Let the bilinear forms \(a_\NC:H^1(\Tri)\times H^1(\Tri)\to \R\) and \(\widetilde b: X \times Y \to \R\) be defined by
\begin{align*}
 a_\NC(v_1, v_2) &\coloneq \int_\Omega \gradNC v_1 \cdot \gradNC v_2 \dx \quad\text{ for }v_1,v_2\in H^1(\Tri),\\
 \widetilde b(x, y)
    &\coloneq
    a_\NC(u, w)
    - \langle t, w \rangle_{\partial\Tri}\quad\text{ for }x = (u, t) \in X, y = w \in Y.
\end{align*}
\begin{lemma}
  \label{lem:discrete_inf_sup_linear}
  The bilinear form \(\widetilde b: X_h \times Y_h \to \R\) satisfies the discrete inf-sup condition
  \begin{equation}
  \label{eq:discrete_inf_sup_linear}
    0 < \widetilde \beta_h \coloneq
    \inf_{\xi_h \in \sphere{X_h}}
    \sup_{\eta_h \in \sphere{Y_h}}
    \widetilde b(\xi_h, \eta_h).
  \end{equation}
\end{lemma}
\begin{proof}
  The proof follows the arguments from \cite[Thm.~3.5]{MR3279489} for the bilinear form \(\widetilde b\) in the
  lowest-order scheme at hand.
  \qed
\end{proof}
\begin{proof}[of Theorem~\ref{thm:reduced_formulation}]
  (a)\; Since
  \(b'(x_h; 0, s_h, v_h) = -\langle s_h, v_h \rangle_{\partial\Tri} =
  0\) for all \(s_h \in P_0(\Edges)\), \(v_h \in CR^1_0(\Tri)\).
  Then,~\eqref{eq:model_mixed} reduces to~\eqref{eq:model_reduced}.

  (b)\; Conversely, suppose \((u_h, v_h)\) solves~\eqref{eq:model_reduced}, then the
  second equation in~\eqref{eq:model_mixed} follows from the second equation
  in~\eqref{eq:model_reduced} and \(v_h \in CR^1_0(\Tri)\).
  The first equation in~\eqref{eq:model_reduced} leads to the first equation
  in~\eqref{eq:model_mixed} for any \(t_h \in P_0(\Edges)\) and test functions in
  \(CR^1_0(\Tri)\).
  In other words, the linear functional
  \[
    \Lambda_h
    \coloneq
    a(v_h, \pbullet)
    + \int_\Omega \sigma(\grad u_h) \cdot \gradNC \pbullet \dx - F \in Y_h^*
  \]
  vanishes on \(CR^1_0(\Tri) \subset \ker \Lambda_h\).
  It remains to show that there exists \(t_h \in P_0(\Edges)\) with
  \(\langle t_h, \pbullet \rangle_{\partial\Tri} = \Lambda_h\), because then
  \((u_h, t_h, v_h)\) solves~\eqref{eq:model_mixed}.
  To prove the existence of such a \(t_h\) for \(\Lambda_h \in Y_h^*\) with
  \(CR^1_0(\Tri) \subset \ker \Lambda_h\), recall the bilinear form \(\widetilde b\) from
  Lemma~\ref{lem:discrete_inf_sup_linear} with discrete inf-sup
  condition~\eqref{eq:discrete_inf_sup_linear} and consider the linear problem that seeks \((u_h,t_h,v_h)\in X_h\times Y_h\) with
  \begin{equation}
    \tag{L}
    \label{eq:linear_auxiliary}
    \begin{split}
      a_\NC(v_h, w_1) + \widetilde b(u_h,t_h,w_1)
      &=
      -\Lambda_h(w_1)
      \quad\text{for all } w_1 \in Y_h,\\
      \widetilde b(w_C,s_0,v_h)
      &=0
      \quad\text{for all } (w_\Conf, s_0) \in X_h.
    \end{split}
  \end{equation}
  Since \(\langle r_0, v_1 \rangle_{\partial\Tri} = 0\) for \(v_1\in P_1(\Tri)\) and for all
  \(r_0 \in P_0(\Edges)\) implies \(v_1 \in CR^1_0(\Tri) \subset
  P_1(\Tri)\), the kernel
  \[
    Z_h
    \coloneq
    \lbrace v_1 \in P_1(\Tri) \,|\, \widetilde b(x_h, v_1) = 0
    \text{ for all } x_h \in X_h \rbrace
  \]
  of \(\widetilde b\)
  consists of particular Crouzeix-Raviart functions, \(Z_h \subset
  CR^1_0(\Tri)\), and the discrete Friedrichs inequality \cite[p.\
  301]{MR2373954}
  shows that \(a_\NC\) is \(Z_h\)-elliptic.
  Hence, the Brezzi splitting lemma \cite[Thm.~4.3 in Ch.~III]{MR2322235} applies
  to the
  linear system
  \eqref{eq:linear_auxiliary} and \eqref{eq:linear_auxiliary} has a unique solution
  \((u_h, t_h, v_h) \in X_h \times Y_h\).
  The test of the first equation in~\eqref{eq:linear_auxiliary} with
  \(w_1 \in CR^1_0(\Tri) \subset \ker \Lambda_h\) shows \(a_\NC(v_h + u_h,
  \pbullet) = 0\) in \(CR^1_0(\Tri)\). The second equation
  in~\eqref{eq:linear_auxiliary} implies \(v_h \in CR^1_0(\Tri)\) and this
  proves \(v_h = - u_h\). This leads to \(\langle t_h, \pbullet
    \rangle_{\partial \Tri} = \Lambda_h\) in \(P_1(\Tri)\).
  The uniqueness of \(t_h\) follows from the fact that \(\langle t_h, \pbullet \rangle_{\partial\Tri} = 0\) in \(P_1(\Tri)\) implies \( t_h=0\).
  \qed
\end{proof}

\subsection{Least-squares formulation}
\label{sec:least-squares}
Recall $S_0\in P_0(\Tri;\R^{n\times n})$ and $H_0:L^2(\Omega)\rightarrow
P_0(\Tri;\R^n)$ from \eqref{eq:S0H0} to define an equivalent least-squares
formulation.
\begin{theorem}[dPG is LS]\label{thm:dpg_as_ls}
Any $x_h=(u_C,t_0)\in X_h$ and $p_\RT\in RT_0(\Tri)$ with $p_\RT\cdot\nu= t_0$ in \(\partial\Tri\) satisfy
\begin{equation}\label{eq:weighted_ls}
  \begin{aligned}
  \Norm{F-b(x_h;\pbullet)}{Y_h^*}^2 &= \NormLzDom{( I_{n\times n}+S_0) ^{-1/2}\big(\Pi_0 p_\RT - \sigma( \nabla  u_C)  + H_0 f\big)}^2\\
  &\quad+ \NormLzDom{\Pi_0 f + \Div p_\RT}^2.
  \end{aligned}
\end{equation}
Consequently, any solution $x_h=(u_C,t_0)\in X_h$ to \eqref{eq:model_mixed} and $p_\RT\cdot\nu= t_0$ in \(\partial\Tri\) from Remark~\ref{rem:identification_RT_P0Edges} minimizes the weighted least-squares functional \eqref{eq:weighted_ls}.
\end{theorem}
\begin{proof}
  Let $v_1 \in P_1(\Tri) \equiv Y_h$ be the Riesz representation of
  $b(x_h;\pbullet)-F\in Y_h^*$, i.e., any $ w_1 \in P_1(\Tri)$
  satisfies
  \begin{align*}
    \changed{a( v_1 , w_1 )} = b(x_h; w_1 )-F( w_1 ).
  \end{align*}
The substitution of $ t_0 =p_\RT\cdot\nu$ based on the isometry in Remark~\ref{rem:identification_RT_P0Edges} and an integration by parts lead to
\begin{align*}
 b(x_h; w_1 )-F( w_1 )
&= \int_\Omega (\sigma(\nabla  u_C) -p_\RT) \cdot \gradNC  w_1 \dx - \int_\Omega (f+\Div p_\RT)  w_1 \dx.
\end{align*}
With $ w_1 =\Pi_0  w_1 + \gradNC w_1 \cdot (\pbullet-\operatorname{mid}(\Tri))$, this results in
\begin{align*}
  &\int_\Omega \Pi_0 v_1 \Pi_0 w_1 \dx
  + \int_\Omega ( I_{n\times n}+S_0)\gradNC v_1 \cdot\gradNC w_1 \dx\\
  &\qquad=
  \int_\Omega (\sigma(\nabla  u_C) -\Pi_0 p_\RT)\cdot \gradNC  w_1 \dx - \int_\Omega (\Pi_0 f+\Div p_\RT) \Pi_0 w_1 \dx\\
  &\qquad\phantom{{}={}}\qquad
  - \int_\Omega H_0 f\cdot\gradNC w_1 \dx.
\end{align*}
For any $T\in \Tri$, the choices $ w_1 =\chi_T$ and $ w_1 =\chi_T e_k \cdot
(\pbullet - \operatorname{mid}(\Tri))$, $k=1,\dots,n$, show
\begin{align*}
  \Pi_0 v_1  &= - (\Div p_\RT +\Pi_0 f),\\
 ( I_{n\times n}+S_0)\gradNC  v_1  &= \sigma(\nabla  u_C) -\Pi_0 p_\RT - H_0 f.
\end{align*}
The Riesz isometry and $\NormH{ v_1 }{1}{\Omega}^2 = \NormLzDom{\Pi_0 v_1 }^2+\NormLzDom{( I_{n\times n}+S_0)^{1/2}\gradNC v_1 }^2$ for any \(v_1\in P_1(\mathcal{T})\) conclude the proof.
  \qed
\end{proof}

\section{Mathematical analysis of dPG for the model problem}\label{sec:analysis}
This section analyses the low-order dPG method presented in
Section~\ref{sec:model_problem} and proves an a~posteriori result next to
the existence of a solution and applies the abstract framework from
Section~\ref{sec:abstract_dpg}.
Recall the discrete spaces \(X_h \coloneq S^1_0(\Tri) \times P_0(\Edges)\),
\(Y_h \coloneq P_1(\Tri)\), and the nonlinear map from \eqref{eq:def_b}.

\subsection{Well-posedness}
\changed{This subsection is devoted to the equivalence of the dPG
residuals and the errors.}
For \(q_\RT\in RT_0(\Tri)\) and \(v_C\in S^1_0(\Tri)\), the isomorphism
between \(RT_0(\Tri)\) and \(P_0(\Edges)\) from
Remark~\ref{rem:identification_RT_P0Edges} leads to the abbreviation
\(b(v_C,q_\RT;\pbullet) \coloneq
b((v_C,(q_\RT\cdot\nu_T)_{T\in\Tri});\pbullet)\).
\changed{Recall the energy norm \(\vvvert \pbullet
\vvvert = \Vert \pbullet \Vert_{L^2(\Omega)}\) in \(H^1_0(\Omega)\).}

\begin{theorem}\label{thm:normequivalence}
 The exact solution $u\in H^1_0(\Omega)$ to the model problem \eqref{eq:EulerLagrange} with stress $ p\coloneq \sigma(\nabla u)\in H(\Div,\Omega)$ and any discrete $(v_C, q_\RT)\in S^1_0(\Tri)\times RT_0(\Tri)$ satisfy the equivalence
\begin{align*}
 \NormHDiv{ p- q_\RT}{\Omega}^2 + \NormEnergy{u-v_C}^2 &\approx \Norm{F-b(v_C,q_\RT;\pbullet)}{Y_h^*}^2\\
 &\phantom{{}={}}\quad+ \NormLzDom{(1-\Pi_0) f}^2 +  \NormLzDom{(1-\Pi_0)q_\RT}^2.
\end{align*}
\end{theorem}
The proof is based on a lemma on the nonlinear least-squares formulation.
The related least-squares formulation is associated with the nonlinear
residual \(\Res(f;\pbullet): H(\Div,\Dom) \times H^1_0(\Dom) \to L^2(\Dom)
\times L^2(\Dom; \R^n)\) for the first-order system of
\eqref{eq:EulerLagrange} and defined,  for \((p,u)\in H(\Div,\Dom)\times H^1_0(\Dom)\), by
\[
  \Res(f;p,u) \coloneq
  (f+\Div p,  p - \sigma(\grad u)).
\]

\begin{lemma}
  \label{lem:equivalence}
  Any \(( p,u),( q,v)\in H(\Div,\Dom) \times H^1_0(\Dom)\) satisfy
  \[
    \NormLzDom{\Res(f; p,u) - \Res(f; q,v)}^2
    \approx
    \NormHDivDom{ p -  q}^2
    + \NormEnergy{u - v}^2.
  \]
\end{lemma}

\begin{proof}
  Following \cite[Thm.~4.4]{MR3264569}, the fundamental
  theorem of calculus shows
  \begin{align*}
    \Res(f; q,v) - \Res(f; p,u)
    &=
    \int_0^1 \frac{\mathrm{d}}{\mathrm{d} s} \Res(f; p + s( q
    -  p), u + s(v -u)) \ds\\
    &=
    \int_0^1 \Res'( p + s( q -  p), u + s(v -u);
     q -  p, v-u) \ds.
  \end{align*}
  For \(x\in\Dom\) and \(0\leq s\leq 1\), define \(F(s)\coloneq \grad u(x) + s\grad(v -u)(x)\) and
  \[
   M(x) \coloneq \int_0^1 \Big(\phi(\abs{F(s)}) I_{n\times n} +
  \phi'(\abs{F(s)}) \frac{F(s) \otimes F(s)}{\abs{F(s)}}\Big) \ds.
  \]
  Then
  \[
   \NormLzDom{\Res(f; q,v) - \Res(f; p,u)}^2 = \NormLzDom{\Div ( q- p)}^2 + \NormLzDom{ q- p-M\grad (v-u)}^2.
  \]
  Since the assumptions on $\phi$ show that \(M\in L^2(\Dom;\R^{n\times n})\)
  is pointwise symmetric and positive definite with eigenvalues in the
  real compact interval \([\gamma_1,\gamma_2]\subset (0,\infty)\), the triangle inequality shows
  \[
   \NormLzDom{\Res(f; q,v) - \Res(f; p,u)}^2 \leq
   2\max\{1,\gamma_2^2\}\big(\NormHDivDom{ q- p} + \NormEnergy{v-u}^2\big).
  \]

  For the reverse estimate, the positive definiteness of \(M\) provides the
  unique existence of a solution \(\alpha\in H^1_0(\Dom)\) to the weighted problem
  \[
    \int_\Omega M\grad\alpha \cdot \grad\gamma\dx = \int_\Omega (q - p) \cdot
    \grad\gamma\dx
    \text{ for any }\gamma\in H^1_0(\Omega).
  \]
  An integration by parts shows \(r\coloneq q-p-M\nabla \alpha\in H(\Div,\Dom)\) with \(\Div r = 0\). The Friedrichs inequality with
  constant \(C_\Friedrichs\) (i.e. \(\NormLzDom{\alpha}\leq C_\Friedrichs \vvvert \alpha\vvvert\)) implies
  \[\NormLzDom{M^{1/2}\grad\alpha}^2 =\int_\Dom
    \Div( p- q)\alpha\dx \leq C_\Friedrichs/\gamma_1
    \NormLzDom{\Div( q- p)} \NormLzDom{M^{1/2}\grad\alpha}.\]
  The orthogonality of \(\grad H^1_0(\Dom)\) and \(H(\Div,\Dom)\cap\{\Div=0\}\) in
  \(L^2(\Dom)\) shows
  \begin{align*}
    \NormLzDom{ q- p}^2 &= \NormLzDom{M\grad\alpha + r}^2
    \leq \gamma_2 \NormLzDom{M^{1/2}\grad\alpha + M^{-1/2}r}^2\\
    &= \gamma_2 \NormLzDom{M^{1/2}\grad\alpha}^2
    +\gamma_2 \NormLzDom{M^{-1/2}r}^2.
  \end{align*}
  The two previous displayed inequalities, the triangle inequality, and the abbreviation \(e\coloneq v-u\) yield
  \begin{align*}
   \NormHDivDom{ q- p}^2 + \NormEnergy{e}^2 &\leq (2+\gamma_2) \NormLzDom{M^{1/2}\grad\alpha}^2+ \NormLzDom{\Div( q- p)}^2\\
   &\quad+ \max\{2,\gamma_2\} \big(\NormLzDom{M^{1/2}\grad(\alpha-e)}^2 + \NormLzDom{M^{-1/2}r}^2\big) \\
   &\leq ((2+\gamma_2)C_\Friedrichs^2/\gamma_1^2+1)\NormLzDom{\Div( q- p)}^2 \\
   &\quad+ \max\{2,\gamma_2\}/\gamma_1 \NormLzDom{ q- p-M\grad e}^2.
   \quad\qed
  \end{align*}
\end{proof}

\begin{proof}[of Theorem~\ref{thm:normequivalence}]
 Since $\Res(f;p,u)=0$, Lemma~\ref{lem:equivalence} with $(q,v)\coloneq (q_\RT,v_C)$ shows
\[
    \NormHDivDom{ p -  q_\RT}^2
    + \NormEnergy{u - v_C}^2
    \approx
    \NormLzDom{f+\Div q_\RT}^2 + \NormLzDom{q_\RT - \sigma(\grad v_C)}^2.
  \]
The $L^2$-orthogonality of $(1-\Pi_0) q_\RT$ and $(1-\Pi_0) f$ onto piecewise constants implies
\begin{align*}
  &\Vert f+\Div q_\RT \Vert_{L^2(\Omega)}^2
  + \Vert q_\RT - \sigma(\grad v_C) \Vert_{L^2(\Omega)}^2\\
  &\qquad=
  \Vert \Pi_0 f+\Div q_\RT \Vert_{L^2(\Omega)}^2
  + \Vert (1-\Pi_0) q_\RT \Vert_{L^2(\Omega)}^2\\
  &\qquad\phantom{{}={}}
  + \Vert (1-\Pi_0) f \Vert_{L^2(\Omega)}^2
  + \Vert \Pi_0 q_\RT-\sigma(\nabla v_C) \Vert_{L^2(\Omega)}^2.
\end{align*}
The triangle inequality and the estimates of Remark~\ref{rem:properties_S0}
result in
\begin{align*}
 \NormHDivDom{p&- q_\RT}^2 + \NormEnergy{u- v_C}^2\\
 &\lesssim    \NormLzDom{\Pi_0 f+\Div q_\RT}^2  +   \NormLzDom{(1-\Pi_0) q_\RT}^2 +  \NormLzDom{(1-\Pi_0) f}^2\\
 &\quad+  \NormLzDom{H_0 f}^2  +  \NormLzDom{\Pi_0  q_\RT-\sigma(\nabla v_C)+H_0 f}^2\\
  &\lesssim    \NormLzDom{\Pi_0 f+\Div q_\RT}^2  +   \NormLzDom{(1-\Pi_0) q_\RT}^2 +  \NormLzDom{(1-\Pi_0) f}^2\\
 &\quad+  \NormLzDom{\Pi_0  q_\RT-\sigma(\nabla v_C)+H_0 f}^2.
\end{align*}
Recall that \(S_0\) is pointwise positive semi-definite, hence \(I_{n\times n}+S_0\) is positive definite and Remark~\ref{rem:properties_S0} also proves
\[
 \NormLzDom{\Pi_0  q_\RT-\sigma(\nabla v_C)+H_0 f}^2 \approx \NormLzDom{(I_{n\times n}+S_0)^{-1/2}(\Pi_0  q_\RT-\sigma(\nabla v_C)+H_0 f)}^2.
\]

The proof of the converse estimate utilizes the last estimate and the triangle inequality to show
\begin{align*}
 \NormLzDom{(I_{n\times n}+S_0&)^{-1/2}(\Pi_0  q_\RT-\sigma(\nabla v_C)+H_0 f)}^2 +  \NormLzDom{(1-\Pi_0) f}^2\\
&\lesssim \NormLzDom{\Pi_0  q_\RT-\sigma(\nabla v_C)}^2  +\NormLzDom{H_0 f}^2 +\NormLzDom{(1-\Pi_0) f}^2\\
&\lesssim \NormLzDom{\Pi_0  q_\RT-\sigma(\nabla v_C)}^2 +\NormLzDom{(1-\Pi_0) f}^2.
\end{align*}
This and the aforementioned orthogonalities imply
\begin{align*}
 \NormLzDom{(I_{n\times n}+S_0&)^{-1/2}(\Pi_0  q_\RT-\sigma(\nabla v_C)+H_0 f)}^2 +  \NormLzDom{(1-\Pi_0) f}^2\\
&\quad + \NormLzDom{(1-\Pi_0) q_\RT}^2 + \NormLzDom{\Pi_0 f+\Div q_\RT}^2 \\
&\lesssim \NormLzDom{ q_\RT-\sigma(\nabla v_C)}^2 + \NormLzDom{f+\Div q_\RT}^2\\
&\lesssim \NormHDivDom{p- q_\RT}^2 + \NormEnergy{u-v_C}^2.
\quad\qed
\end{align*}
\end{proof}

\subsection{Existence and uniqueness of discrete solutions}
The existence of discrete solutions follows from variational arguments, while their uniqueness is fairly open.
\begin{proposition}
  \label{prop:discrete_existence}
  The discrete problem~\eqref{eq:model_dpg} has a solution.
\end{proposition}
\begin{proof}
  The proof follows with the direct method in the calculus of variations
  and, in the present case of finite dimensions, from the global minimum of
  a continuous \changed{functional} on a compact set from the growth condition
  \begin{equation}
    \label{eq:growthcondition}
    \lim_{\Vert \xi_h \Vert_X \rightarrow\infty} \Vert F-B\xi_h
    \Vert_{Y_h^*} = \infty.
  \end{equation}
The latter property follows from Theorem~\ref{thm:normequivalence} up to
some \changed{perturbation} terms. Theorem~\ref{thm:dpg_as_ls} shows
\begin{align*}
 \NormLzDom{(1-\Pi_0)q_\RT} &\leq \NormLzDom{h_\Tri \Pi_0 f} + h_{\textup{max}}\NormLzDom{\Pi_0 f + \Div q_\RT}\\
 &\leq  \NormLzDom{h_\Tri \Pi_0 f} + h_{\textup{max}}\Vert F-B\xi_h \Vert_{Y_h^*}.
\end{align*}
The combination with Theorem~\ref{thm:normequivalence} shows that the right-hand side of Theorem~\ref{thm:normequivalence} is bounded from above by
\begin{align*}
 &(1+2h_{\textup{max}}^2)  \Vert F-B\xi_h \Vert_{Y_h^*}^2 + 2 h_{\textup{max}}^2
 \NormLzDom{\Pi_0 f}^2 + \NormLzDom{(1-\Pi_0) f}^2\\
 &\qquad\leq  (1+2h_{\textup{max}}^2)  (\Vert F-B\xi_h \Vert_{Y_h^*} + \NormLzDom{f}).
\end{align*}
Hence the left-hand side in Theorem~\ref{thm:normequivalence} is controlled by this and so
\[
 \Vert x-\xi_h\Vert_X \lesssim \Vert F-B\xi_h \Vert_{Y_h^*} + \NormLzDom{f}.
\]
Since \(f\) and \(x\) are fixed, this implies~\eqref{eq:growthcondition} and
concludes the proof.
  \qed
\end{proof}

The uniqueness of the exact solution \((u,t)\) on the continuous level does
not imply the uniqueness of discrete solutions. There is, however, a
sufficient condition for a global unique discrete solution. Notice that
\(v_h=v=0\) on the continuous level \(h=0\) satisfies \eqref{eq:uniqueness_criterion}.
\begin{theorem}[a~posteriori uniqueness]
\label{thm:a_post_unique}
 Suppose that \((u_h,v_h)\in S^1_0(\Tri)\times CR^1_0(\Tri)\) solves \eqref{eq:model_reduced} with \(\Grad\sigma\in C(\R^n;\R^{n\times n}_{\operatorname{sym}})\) globally Lipschitz continuous and
 \begin{equation}
  \label{eq:uniqueness_criterion}
  \operatorname{Lip}(\Grad\sigma)(1+C_\Friedrichs^2)/\gamma_1^2\NormL{\gradNC
  v_h}{\infty}{\Dom} <1
 \end{equation}
with the Friedrichs constant \(C_\Friedrichs\) from \(\NormLzDom{\pbullet}\leq C_\Friedrichs \vvvert \pbullet\vvvert\) in \(H^1_0(\Dom)\).
Then \eqref{eq:model_reduced} has exactly one solution \((u_h,v_h)\in S^1_0(\Tri)\times CR^1_0(\Tri)\).
\end{theorem}
\begin{proof}
 Suppose that \((\widetilde u_h,\widetilde v_h)\in S^1_0(\Tri)\times CR^1_0(\Tri)\) solves \eqref{eq:model_reduced} as well and so
 \begin{align*}
  a(v_h,v_h) &= F(v_h) - \int_\Dom \sigma(\nabla u_h)\cdot \gradNC v_h \dx\\
  &= a(\widetilde v_h, v_h) + \int_\Dom (\sigma(\nabla \widetilde u_h) - \sigma(\nabla u_h))\cdot \gradNC v_h \dx.
 \end{align*}
This \changed{and} the second equation of \eqref{eq:model_reduced} imply
\begin{align*}
  a(v_h-\widetilde v_h,v_h) = \int_\Dom \gradNC v_h \cdot (\sigma(\nabla
    \widetilde u_h) - \sigma(\nabla u_h) - \Grad\sigma(\nabla
  u_h)\nabla(\widetilde u_h - u_h) \changed{)}\dx.
\end{align*}
 Since \(\sigma\in C^1(\R^n)\) is bounded and \(\Grad\sigma\) Lipschitz continuous, any \(A,B\in\R^n\) with \(F(s)\coloneq (1-s)A+sB\) for \(0\leq s\leq 1\) satisfy
 \begin{align*}
  \vert \sigma(B)&-\sigma(A) - \Grad\sigma(A)(B-A)\vert = \vert \int_0^1 (\Grad\sigma(F(s))- \Grad\sigma(A))(B-A)\ds\vert\\
  &\leq \operatorname{Lip}(\Grad\sigma) \vert B-A\vert \int_0^1 \vert F(s) - A\vert \ds= \frac12 \operatorname{Lip}(\Grad\sigma)\vert B-A\vert^2.
 \end{align*}
With \(A=\nabla u_h(x)\) and \(B=\nabla\widetilde u_h(x)\) for a.e. \(x\) and an integration over \(\Dom\), this leads in the preceding identity to
\begin{equation}\label{eq:estimate_a_vh}
 a(v_h-\widetilde v_h,v_h) \leq \frac12\operatorname{Lip}(\Grad\sigma)\NormL{\gradNC v_h}{\infty}{\Dom} \vvvert u_h-\widetilde u_h\vvvert^2.
\end{equation}
The discrete solutions of \eqref{eq:model_reduced} lead to the same minimal discrete residual norm and hence
\[
 \Vert v_h\Vert_{Y_h} = \Vert F- b(u_h,t_h;\pbullet)\Vert_{Y_h^*} = \Vert F- b(\widetilde u_h,\widetilde t_h;\pbullet)\Vert_{Y_h^*} =\Vert\widetilde v_h\Vert_{Y_h}.
\]
This shows \(a(v_h - \widetilde v_h, v_h + \widetilde v_h)=0\) and the combination with \eqref{eq:estimate_a_vh} is
\begin{equation}\label{eq:estimate_error_vh}
\begin{aligned}
 \Vert v_h - \widetilde v_h \Vert_{Y_h}^2 &= 2a(v_h - \widetilde v_h, v_h) - a(v_h - \widetilde v_h, v_h + \widetilde v_h)\\
 &\leq \operatorname{Lip}(\Grad\sigma)\NormL{\gradNC v_h}{\infty}{\Dom}\vvvert u_h - \widetilde u_h\vvvert^2.
 \end{aligned}
\end{equation}
On the other hand, \(\Grad \sigma(A) \in \R^{n\times n}_\sym\)
  has eigenvalues in the compact interval \([\gamma_1, \gamma_2] \subset
  (0, \infty)\) and so, for all \(A, B \in \R^n\),
  \begin{equation}
    \label{eq:convexity}
    \begin{split}
      \gamma_1 \abs{A-B}^2
      \leq \int_0^1 (A-B) \cdot \Grad\sigma(B + s(A-B)) (A-B) \ds\\
      = (\sigma(A) - \sigma(B))\cdot (A-B)
      \leq \gamma_2 \abs{A-B}^2.
    \end{split}
  \end{equation}
With \(A=\nabla u_h(x)\) and \(B=\nabla \widetilde u_h(x)\), and an integration over a.e. \(x\in\Dom\), this shows
\[
 \gamma_1 \vvvert u_h - \widetilde u_h \vvvert^2 \leq \int_\Dom (\sigma(\nabla u_h)-\sigma(\nabla\widetilde u_h))\cdot \nabla(u_h-\widetilde u_h) \dx.
\]
The first identity in \eqref{eq:model_reduced} for \((u_h,v_h)\) and \((\widetilde u_h,\widetilde v_h)\), respectively, results in
\begin{align*}
 \int_\Dom (\sigma(\nabla u_h)-\sigma(\nabla\widetilde u_h))\cdot \nabla(u_h-\widetilde u_h) \dx &= a(\widetilde v_h - v_h, u_h - \widetilde u_h)\\
 &\leq \Vert v_h - \widetilde v_h\Vert_{Y_h} \sqrt{1+C_\Friedrichs^2} \vvvert u_h - \widetilde u_h \vvvert.
\end{align*}
The combination with the previous inequality shows
\begin{equation}
\label{eq:error_u}
 \gamma_1 \vvvert u_h - \widetilde u_h\vvvert \leq \sqrt{1+C_\Friedrichs^2} \Vert v_h - \widetilde v_h\Vert_{Y_h}.
\end{equation}
The substitution in \eqref{eq:estimate_error_vh} results in
\[
 \Vert v_h - \widetilde v_h\Vert_{Y_h}^2 \leq \operatorname{Lip}(\Grad\sigma)(1+C_\Friedrichs^2)/\gamma_1^2\NormL{\gradNC v_h}{\infty}{\Dom}  \Vert v_h - \widetilde v_h\Vert_{Y_h}^2.
\]
This and \eqref{eq:uniqueness_criterion} show \(v_h = \widetilde v_h\). Then \eqref{eq:error_u} implies \(u_h = \widetilde u_h\).
  \qed
\end{proof}

\subsection{Best-approximation}
For any \(v\in H^1(\Tri)\), the \emph{nonconforming interpolation}
\(I_\NC^\loc v \in P_1(\Tri)\) is defined, on each triangle \(T\in\Tri\),
by piecewise linear interpolation of the values
\begin{equation}
  \label{eq:nonconf_interpolation}
  (I_\NC^\loc v)(\Mid(E)) \coloneq \intmean_E v\restrict{T} \ds
\end{equation}
 at the midpoints of the sides \(E\in\Edges(T)\).
\begin{proposition}
  \label{prop:fortin_model}
  The operator \(\Pi \coloneq I_\NC^\loc\) satisfies
  Hypothesis~\ref{hyp:fortin}.
\end{proposition}

\begin{proof}
  Given \(v \in H^1(\Tri)\), set \(v_h \coloneq I_\NC^\loc v \in
  P_1(\Tri)\). For every \(K\in\Tri\), an integration by parts leads to
  \[
    \grad \big(v_h\restrict{K}\big) = \int_{\partial K} v_h \cdot \nu_K\ds = \int_{\partial K} v\cdot \nu_K \ds=
    \int_K \grad v \dx \Big/ \abs{K}.
  \]
  Since \(\grad w_C\in P_0(\Tri;\R^n)\) and \eqref{eq:dsigma} shows \(\Grad \sigma(\grad u_C) \in P_0(\Tri;\R^{n\times n})\) for all \(w_C,u_C
  \in S^1_0(\Tri)\), this implies
  \[
  \int_\Omega \grad w_C \cdot (\Grad \sigma(\grad u_C) \gradNC (v-v_h)) \dx=0.
  \]
  Moreover, \eqref{eq:nonconf_interpolation} guarantees that any \(s_0\in P_0(\Edges)\) satisfies
  \[
    \dual{s_0, v-v_h}_{\partial\Tri} = 0.
  \]
  Consequently, any \(x_h=(u_C,t_0)\in X_h\) and \(\xi_h=(w_C,s_0)\in X_h\) satisfy
  \begin{align*}
   b'(x_h; \xi_h,v-v_h) &= \int_\Omega \grad w_C \cdot (\Grad \sigma(\grad u_C) \gradNC (v-v_h)) \dx - \dual{s_0, v-v_h}_{\partial\Tri}= 0.
   \quad\qed
  \end{align*}
\end{proof}

The estimates for the function \(\Grad\sigma\) from
Subsection~\ref{sec:convex_energy_minimization} lead to an explicit generic constant for the
best-approximation estimate from
Theorem~\ref{thm:ba} without any local hypothesis.

\begin{theorem}[best-approximation]
  \label{thm:best_approximation}
  Let \(x = (u,t) \in X\) be the unique solution to
\(B(x) = F\) for the nonlinear map \(B\) from \eqref{eq:def_b} in Section~\ref{sec:dpg_discretization}.
  Any discrete solution \((u_h,t_h;v_h) \in X_h\times Y_h\) to~\eqref{eq:model_dpg}
  satisfies
  \begin{align*}
    \vvvert u - u_h \vvvert + \Vert v_h \Vert_Y
    &\lesssim
    \inf_{u_\Conf \in S^1_0(\Tri)} \vvvert u - u_\Conf \vvvert+ \inf_{t_0 \in P_0(\Edges)}
    \Vert t - t_0 \Vert_{H^{-1/2}(\partial \Tri)}.
  \end{align*}
\end{theorem}
\begin{proof}
  Given the best-approximation \(x_h^*=(u_h^*,t_h^*)\in X_h\) to \((u,t)\) in \(X_h\) and let \(A=\nabla u_h^*(x)\) and \(B=\nabla u_h(x)\) in \eqref{eq:convexity} and integrate over a.e. \(x\in\Dom\). Then
  \begin{equation*}
    \begin{split}
      \gamma_1 \vvvert u^*_h - u_h \vvvert^2
      &\leq
      \int_\Omega \big( \sigma(\grad u_h^*) - \sigma(\grad u_h) \cdot \grad
        (u_h^* - u_h) \big) \dx\\
      &= b(x_h^*; u_h^* - u_h) - b(x_h; u_h^* - u_h).
    \end{split}
  \end{equation*}
Since \(b(u,t;\pbullet)=F\), the last term is equal to
  \begin{equation}
    \label{eq:error_estimate}
    \begin{split}
      F(u_h^* &- u_h) - b(x_h; u_h^* - u_h) + b(x_h^*; u_h^* - u_h)- b(u; u_h^* - u_h)\\
      &= a(v_h, u_h^* - u_h)
      + \int_\Omega \big( \sigma(\grad u_h^*) - \sigma(\grad u) \big) \cdot
      \grad (u_h^* - u_h) \dx.
    \end{split}
  \end{equation}
  The Lipschitz continuity from Remark~\ref{rem:lipschitz} leads to
  \begin{equation}
    \label{eq:error_sigma}
    \Vert \sigma(\grad u) - \sigma(\grad u_h^*) \Vert_{L^2(\Omega)}
    \leq
    \gamma_2 \vvvert u - u_h^* \vvvert
  \end{equation}
  and the last term in~\eqref{eq:error_estimate} is controlled by
  \begin{equation}
  \label{eq:est_int_sigma}
    \int_\Omega \big( \sigma(\grad u_h^*) - \sigma(\grad u) \big) \cdot \grad
    (u_h^* - u_h) \dx
    \leq
    \gamma_2 \vvvert u - u_h^* \vvvert \, \vvvert u_h^* - u_h \vvvert.
  \end{equation}
  Since \(x_h\) is a global discrete minimizer,
  \begin{align*}
    \Vert v_h \Vert_Y
    &=
    \Vert F- b(x_h;\pbullet) \Vert_{Y_h^*}
    \leq
   \Vert F- b(x_h^*;\pbullet) \Vert_{Y_h^*}
    =
    \Vert b(x; \pbullet) - b(x_h^*; \pbullet) \Vert_{Y_h^*}.
  \end{align*}
 The Lipschitz continuity \eqref{eq:error_sigma} of \(\sigma\) and the structure of the map \(b\) from \eqref{eq:def_b} show that the last term is \(\leq \gamma_2 \vvvert u - u_h^* \vvvert + \Vert t - t_h^* \Vert_{H^{-1/2}(\partial\Tri)}\).
  The combination of \(\Vert v_h \Vert_Y\leq \gamma_2 \vvvert u - u_h^* \vvvert + \Vert t - t_h^* \Vert_{H^{-1/2}(\partial\Tri)}\) with~\eqref{eq:error_estimate} and~\eqref{eq:est_int_sigma} shows
\begin{align*}
    \gamma_1\vvvert u_h^* - u_h \vvvert &\leq \sqrt{1+C_\Friedrichs^2}\Vert v_h \Vert_Y  + \gamma_2 \vvvert u - u_h^* \vvvert.
  \end{align*}
 A triangle inequality
  concludes the proof with explicit constants
\begin{align*}
    \vvvert u - u_h \vvvert &\leq
    (1+\gamma_2(1+\sqrt{1+C_\Friedrichs^2})\big/\gamma_1)
    \inf_{u_\Conf \in S^1_0(\Tri)} \vvvert u - u_\Conf \vvvert\\
    &\phantom{{}\leq{}}+ \gamma_2 \sqrt{1 +C_\Friedrichs^2} \big/\gamma_1
    \inf_{t_0 \in P_0(\Edges)}
    \Vert t - t_0 \Vert_{H^{-1/2}(\partial \Tri)}.
    \quad\qed
  \end{align*}
\end{proof}

The following a~posteriori error estimate holds for any discrete
approximation, and even for inexact solve, and generalizes the built-in
error control despite inexact solve of \cite[Thm.~2.1]{MR3215064} to the
nonlinear model problem at hand.
\begin{theorem}[a~posteriori]
  \label{thm:a_posteriori}
  There exist universal constants \(\kappa\approx 1\approx C_{\textup{d}\Friedrichs}\) such that the exact solution \((u,t)\in X\) of \(B(x)=F\) and any discrete \((v_C,s_0)\in
  S^1_0(\Tri)\times P_1(\Tri)\) satisfy
  \[
    \gamma_1^2 \vvvert u - v_C \vvvert^2
    \leq (1+C_{\textup{d}\Friedrichs}^2) \Vert F-b(v_C,s_0;\pbullet) \Vert_{CR^1_0(\Tri)^*}^2 + \kappa^2\NormLzDom{h_\Tri f}^2.
  \]
\end{theorem}
\begin{remark}
 The proof reveals that \(C_{\textup{d}\Friedrichs}\) is the constant in the discrete Friedrichs inequality \cite[p.~301]{MR2373954} \(\Vert\pbullet\Vert \leq C_{\textup{d}\Friedrichs} \Vert \gradNC \pbullet \Vert\) in \(CR^1_0(\mathcal{T})\). The explicit bounds of \(C_{\textup{d}\Friedrichs}\) in \cite{CH17} allow quantitative estimates in \(2\)D and show in particular \(C_{\textup{d}\Friedrichs}\leq 6.24\) for a convex domain with \(\operatorname{diam}(\Dom)\leq 1\) and a triangulation with right isosceles triangles.
\end{remark}
\begin{remark}
  The proof reveals that \(\kappa\) is the constant in interpolation error
  estimate for the nonconforming interpolation operator \(\Vert h_\Tri^{-1}
  (1-I_\NC)v\Vert \leq \kappa \Vert \gradNC (1-I_\NC)v \Vert\) for \(v\in
  H^1(\Omega)\). An estimate with the first positive root $j_{1,1}$ of the
  Bessel function of the first kind in \cite[Thm.~4]{MR3149071} in \(2\)D
  reads $\kappa=(1/48+1/j_{1,1}^2)^{1/2}= 0.29823$.
\end{remark}

\begin{proof}[of Theorem~\ref{thm:a_posteriori}]
  The estimate~\eqref{eq:convexity} with \(A = \grad u(x)\), \(B = \grad
  u_h(x)\), \(e\coloneq u-u_h\), and an integration over a.e. \(x\in\Omega\) leads to
  \begin{equation}
  \label{eq:estimate_error_u}
    \gamma_1 \vvvert e \vvvert^2
    \leq
    \int_\Omega \big( \sigma(\grad u) - \sigma(\grad u_h) \big)
    \cdot \grad e \dx.
  \end{equation}
  Since \((u,t) \in X\) solves \(b(u,t; \pbullet) = F\) in \(Y^*\) and with the nonconforming interpolation operator \eqref{eq:nonconf_interpolation}, this is equal to
  \begin{align*}
    F(e) - \int_\Omega \sigma(\grad u_h) \cdot \grad e \dx
    &= F((1-I_\NC)e) + F(I_\NC e) - \int_\Omega \sigma(\grad u_h) \cdot \gradNC I_\NC e \dx\\
    &= F((1-I_\NC)e) + F(I_\NC e) - b(u_h,t_h;I_\NC e)\\
    &\leq F((1-I_\NC)e) + \Vert F-b(u_h,t_h;\pbullet)\Vert_{CR^1_0(\Tri)^*} \Vert I_\NC e \Vert_{Y_h}.
  \end{align*}
  The interpolation error estimate for the nonconforming interpolation
  operator with constant $\kappa$ \cite[Thm.~4]{MR3149071} yields
  \[
    F((1-I_\NC)e) \leq \kappa \NormLzDom{h_\Tri f} \vvvert e-I_\NC e \vvvert_\NC.
  \]
The discrete Friedrichs inequality \cite[p.~301]{MR2373954} and \(I_\NC e\in CR^1_0(\Tri)\) prove
  \begin{align*}
    \Vert I_\NC e \Vert_{Y_h} &\leq \sqrt{1+C_{\textup{d}\Friedrichs}^2} \vvvert I_\NC e \vvvert_\NC.
  \end{align*}
The Cauchy inequality in \(\R^2\) and the theorem of Pythagoras imply
\begin{align*}
 \gamma_1 \vvvert e\vvvert^2 \leq \big(\kappa^2 \NormLzDom{h_\Tri f}^2 +
 (1+C_{\textup{d}\Friedrichs}^2) \Vert F-
 b(u_h,t_h;\pbullet)\Vert_{CR^1_0(\Tri)^*}^2\big)^{1/2} \vvvert e
 \vvvert_\NC. \quad\qed
\end{align*}
\end{proof}

\subsection{Other nonlinear dPG methods}\label{sec:other_dpg}
This section illustrates the plethora of dPG methodology by introducing the primal mixed, the dual, and the ultraweak dPG method for the nonlinear model problem.
All three methods concern the first-order system of \eqref{eq:EulerLagrange} with the convex function \(\varphi\) and \(\sigma=\Grad (\varphi\circ \vert\pbullet\vert)\) and its dual \(\varphi^*\) so that the  relation \(p=\sigma(\nabla u)\) is equivalent to \(\nabla u =\Grad \varphi^*(\vert p \vert )\sign{p}\) on the continuous level.
Recall the space of functions with piecewise divergence \(H(\Div;\Tri)
\coloneq \prod_{T \in \Tri} H(\Div;T) \) from \cite{MR3279489} as well as the piecewise version \(RT^\NC_k(\Tri)\subset H(\Div;\Tri)\) of \(RT_k(\Tri)\), and the subspace \(S^k_0(\Edges)\equiv S^k_0(\Tri)\restrict{\partial\Tri}\) of
\begin{align*}
  H^{1/2}_0(\partial\Tri)
  &\coloneq
  \{ s = (s_T)_{T\in\Tri}\in  {\textstyle \prod_{T\in\Tri} }
    H^{1/2}(\partial T) \,|\,
    \exists v\in H^1_0(\Dom),\forall T \in \Tri,s_T=(v|_T)|_{\partial
    T}\}.
 \end{align*}
Recall the  {\bf primal nonlinear dPG method} \eqref{eq:dpg} in Section~\ref{sec:dpg_discretization} with \(b\) from \eqref{eq:def_b} and general polynomial degree \(k\geq 0\) and \(m\geq k\) in the discrete spaces
\begin{align*}
 X_h := S^{k+1}_0(\Tri) \times P_k(\Edges)  \text{ and }Y_h := P_{m+1}(\Tri).
\end{align*}
The {\bf primal mixed nonlinear dPG method} departs from a piecewise integration by parts and employs the spaces and discrete subspaces
\begin{align*}
 X&:= L^2(\Omega;\R^n) \times H^1_0(\Tri) \times H^{1/2}(\partial\Tri)\text{ and }Y:= L^2(\Omega;\R^n) \times H^1(\Tri),\\
 X_h &:= P_k(\Tri;\R^n) \times S^{k+1}_0(\Tri) \times P_k(\Edges)  \text{ and }Y_h := P_m(\Tri;\R^n) \times P_{m+1}(\Tri).
\end{align*}
For $(p,u,t)\in X$ and $(q,v)\in Y$, define \eqref{eq:dpg} with \(F(q,v) := (f,v)_{L^2(\Omega)}\) and
\begin{align*}
 b(p,u,t;q,v) &:= (p-\sigma(\nabla u),q)_{L^2(\Omega)}+(p,\gradNC v)_{L^2(\Omega)} -\langle t,v\rangle_{\partial\Tri}.
\end{align*}
%
The {\bf dual nonlinear dPG method} utilizes the spaces and discrete subspaces
\begin{align*}
 X&:= H(\Div;\Omega) \times L^2(\Omega) \times H^{1/2}_0(\partial\Tri)\text{ and }Y:= H(\Div;\Tri) \times L^2(\Omega),\\
 X_h &:= RT_k(\Tri)\times P_k(\Tri) \times S^{k+1}_0(\Edges)\text{ and } Y_h := RT_m^\NC(\Tri) \times P_m(\Tri) .
\end{align*}
For \(F\) as before and $(p,u,s)\in X$ and $(q,v)\in Y$, define \eqref{eq:dpg} with \(b(p,u,s;q,v):=\)
\begin{align*}
 (\Grad \varphi^*(\vert p \vert )\sign{p},q)_{L^2(\Omega)} + (u,\Div_\NC q)_{L^2(\Omega)}-(\Div p,v)_{L^2(\Omega)} -\langle q\cdot\nu, s\rangle_{\partial\Tri}.
\end{align*}
The {\bf ultraweak nonlinear dPG method} utilizes a piecewise integration by parts in both equations of the first-order system and the spaces
\begin{align*}
 X&:= L^2(\Omega;\R^n) \times L^2(\Omega) \times H^{1/2}(\partial\Tri) \times H^{1/2}_0(\partial\Tri)\text{ and } Y:= H(\Div;\Tri) \times H^1(\Tri),\\
 X_h &:= P_k(\Tri;\R^n) \times P_k(\Tri) \times P_k(\Edges) \times S^{k+1}_0(\Edges) \text{ and } Y_h:= RT_m^\NC(\Tri) \times P_{m+1}(\Tri).
\end{align*}
For \(F\) from the above primal mixed method and $(p,u,t,s)\in X$ and $(q,v)\in Y$, define \eqref{eq:dpg} with
\begin{align*}
 b(p,u,t,s;q,v) &:= (\Grad \varphi^*(\vert p \vert )\sign{p},q)_{L^2(\Omega)} + (p,\gradNC v)_{L^2(\Omega)} + (u, \Div_\NC q)_{L^2(\Omega)}\\
 &\quad\quad-\langle q\cdot\nu, s\rangle_{\partial\Tri}-\langle t, v\rangle_{\partial\Tri}.
\end{align*}
The linear version is analysed in \cite{MR2837484,MR3279492,MR3521055,MR3279489}. The four nonlinear dPG methods may be further analysed in the spirit of this section.

\section{Numerical experiments}\label{sec:numerical_experiments}
This section presents numerical experiments with the LS-FEM
of Subsection~\ref{sec:least-squares}.

\subsection{Computational realization}
Given \(f \in L^2(\Omega)\),
\changed{the discrete solution of~\eqref{eq:model_mixed} is}
solved by a Newton scheme with an initial iterate from the solution of
the \changed{scaled} linear Poisson model problem.
Let \(S^1_0(\Tri)\) be endowed with the energy norm
\(\vvvert \pbullet \vvvert\) and \(RT_0(\Tri)\) with
\(\Vert \pbullet \Vert_{H(\Div,\Omega)}\) and let \(\vvvert \pbullet
\vvvert_*\) denote the norm of the dual space of \(S^1_0(\Tri)
\times RT_0(\Tri)\).
The first Fr\'echet derivative \(\Grad\LSfun(f; u_\Conf, p_\RT)\) of
\(\LSfun(f;\pbullet)\) belongs \changed{to} the dual space of \(S^1_0(\Tri)
\times RT_0(\Tri)\).
After at most \(5\) Newton iterations,
every displayed discrete solution \((u_h, p_h)\) in the following
subsections satisfy
\(\vvvert \Grad \LSfun(f; u_h, p_h, \pbullet) \vvvert_{*} = 0\) up to
machine precision.
In the case of successive mesh-refinement, the iteration starts with the
prolongated solution from the coarser triangulation and terminates in at
most \(3\) or \(4\) iterations.

Table~\ref{tab:newton} presents the errors \(\vvvert \Grad\LSfun(f;
u_h^{(j)}, p_h^{(j)}, \pbullet)\vvvert_{*}\) of the Newton iterate
\((u_h^{(j)}, p_h^{(j)})\) for \(j=0,1,\dots,5\) on fixed triangulations of the square
domain from Subsection~\ref{sec:square} and the L-shaped domain from
Subsection~\ref{sec:lshape} with the convex function \(\phi\) from
Example~\ref{ex:density}.a.
The iterations (A) and (B) utilize a uniform triangulation of the square
domain with \(4\,096\) triangles (\(\ndof = 8\,193\)) and an initial
iterate from a Poisson model problem  for (A) and a weighted Poisson problem with constant weight \(2.5\) in (B).
The adaptive mesh of the L-shaped domain with \(3\,450\) triangles (\(\ndof =
6.901\)) in (C) and (D) has been generated by the algorithm from
Subsection~\ref{sec:adaptive_algorithm} below at level \(\ell=12\). Iteration (C) starts with a
weighted Poisson solution with constant factor \(2.5\) and (D) with the
prolongated solution from the previous mesh.

From the very beginning of the Newton iteration, all values in
Table~\ref{tab:newton} provide numerical evidence \changed{for}
Q-quadratic convergence.
\begin{table}
  \centering
  \pgfplotstableset{
  row sep=\\,
  every head row/.style={
    before row={%
      \toprule
    },
    after row=\midrule,
  },
  every last row/.style={
    after row=\bottomrule,
  },
  columns/Niter/.style={
    int detect,
    column name=\(n_{\textup{iter}}\),
  },
  columns/square1/.style={
    column name={(A)},
    sci, sci zerofill, precision=5, dec sep align,
  },
  columns/square25/.style={
    column name={(B)},
    sci, sci zerofill, precision=5, dec sep align,
  },
  columns/lshape25/.style={
    column name={(C)},
    sci, sci zerofill, precision=5, dec sep align,
  },
  columns/lshapeprol/.style={
    column name={(D)},
    sci, sci zerofill, precision=5, dec sep align,
  },
}
\pgfplotstabletypeset[
  columns={Niter, square1, square25, lshape25, lshapeprol},
]{
     Niter           square1           square25           lshape25         lshapeprol\\
         0  1.6743089691e+01   8.7323006599e+00   6.4312459210e+00   1.1698650723e-01\\
         1  2.2012402501e+00   7.6927360410e-02   4.1943882156e-02   1.9409219111e-03\\
         2  1.5987198526e-01   2.1390880185e-04   8.0426293196e-05   3.3007189794e-06\\
         3  9.6152901290e-04   1.7410138364e-09   3.7427345739e-10   1.1744139631e-11\\
         4  4.1172988346e-08   1.1268899597e-14   6.1115591047e-15   6.2648489945e-15\\
         5  1.1366742864e-14   1.0914249116e-14   5.7081872852e-15   5.9358718790e-15\\
         6  1.1113124931e-14   1.1054977509e-14   5.9276016411e-15   5.8299008737e-15\\
         7  1.0910849304e-14   {}                 {}                 6.0397755279e-15\\
         8  1.1449302786e-14   {}                 {}                 {}              \\
}
\medskip

  \caption{\noindent Convergence history of Newton iteration for
  \(4\) representative examples.}
  \label{tab:newton}
\end{table}

In order to investigate the uniqueness of discrete solutions,
the minimal and the maximal eigenvalue \(\lambda_{\textup{min}}\) and
\(\lambda_{\textup{max}}\) of the Hessian matrix
\(\Hess \LSfun(f; u_h, p_h; \pbullet, \pbullet)\)
of the least-squares functional is computed, where \((u_h, q_h) \in X_h\)
and \(\lambda \in \R\) satisfy, for all \((\widetilde v_h, \widetilde q_h)
\in X_h\),
\begin{equation}
\label{eq:gevp}
  \Hess \LSfun(f; u_h, p_h; v_h, q_h, \widetilde v_h,
  \widetilde q_h)
  =
  \lambda \big( a_\NC(v_h, \widetilde v_h) +
    (q_h, \widetilde q_h)_{H(\Div,\Omega}\big).
\end{equation}
The value \(\lambda_{\textup{min}}\) is uniformly bounded from zero for the
examples in the following subsections, so that
every computed discrete solution \((u_h, p_h)\) is a
local minimizer.

For any discrete \changed{approximation} \((u_h, p_h)\),
\changed{Theorem~\ref{thm:dpg_as_ls} and~\ref{thm:a_posteriori}} verify
the a~posteriori error estimator
\(\eta^2(\Tri) \coloneq \LS(f; u_h, p_h) + \Vert h_\Tri f
\Vert_{L^2(\Omega)}^2\)
\changed{even for \emph{inexact solve} in its computation}.
In view of a lacking proof in Subsection~\ref{sec:lshape} below that the
computed discrete solution is in fact a \emph{global discrete minimizer}
(at least up to machine precision), it is only by \changed{this universal}
a~posteriori error control that we know that the computed approximations
converge to the exact solution.

\subsection{Numerical example on square domain}
\label{sec:square}
This subsection considers the nonlinear model problem for
the exact solution
\[
  u(x) \coloneq \cos(\pi x_1/2) \cos(\pi x_2/2)
  \quad\text{for } x \in \Omega \coloneq (-1,1)^2
\]
with homogeneous Dirichlet boundary conditions,
\(f \coloneq -\Div(\sigma(\grad u))\), and \(\phi\) from
Example~\ref{ex:density}.a.
This defines the exact stress function \(p\coloneq \sigma(\grad u) \in H(\Div,\Omega)\).

Figure~\ref{fig:square_convergence} displays the error estimator \(\eta_\ell\coloneq \eta(\Tri_\ell)\) at the
discrete solutions \((u_\ell, p_\ell)\) on each level \(\ell\) of a
sequence of uniform triangulations as well as the error to the exact
solution \((u,p)\).
The reference energy \(E(u) = -5.774337908509\) in the energy difference
\(E(u_\ell) - E(u) \geq \gamma_1 \vvvert u - u_\ell \vvvert^2/2\) has been
approximated by the energies of \(P_1\)-conforming finite element solution
with an Aitken \(\Delta^2\) extrapolation.
The eigenvalues of \eqref{eq:gevp} in all experiments of Figure~\ref{fig:square_convergence} satisfy \(1.597\leq \lambda_{\textup{min}}\leq 1.722\) and \(9.943\leq \lambda_{\textup{max}}\leq 16.128\) and so prove that the discrete solutions are local minimizers.
The parallel graphs confirm the equivalence of the built-in error estimator
\(\Vert y_\ell \Vert_{Y} = (\LSfun(f; u_\ell, p_\ell))^{1/2}\)
with the exact error from Theorem~\ref{thm:normequivalence}.
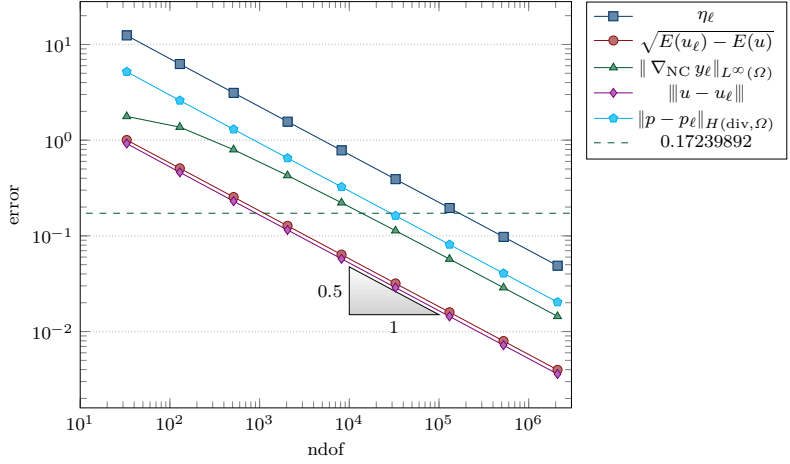
\begin{figure}
  \centering
  \begin{tikzpicture}[scale=0.85]
  \begin{loglogaxis}[convergenceplot,xmin=1e1,xmax=3e6]
    \pgfplotsset{cycle list={
      {HUblue,  mark=square*,   every mark/.append style={fill=HUblue!60!white}},%
      {HUred,   mark=*,         every mark/.append style={fill=HUred!60!white}},%
      {HUgreen, mark=triangle*, every mark/.append style={fill=HUgreen!60!white}},%
      {violet,  mark=diamond*,  every mark/.append style={fill=violet!60!white}},%
      {cyan,    mark=pentagon*, every mark/.append style={fill=cyan!60!white}},%
    }};
    \pgfplotstableread{
      ndof	runtime	acc	it	vmax	eta	energyDiff	errorEnergy	errorHdiv
      33	1.51746000e-01	2.82367705e-15	4	1.77099538e+00	1.24270271e+01	1.00377066e+00	9.22517314e-01	5.18951831e+00
      129	2.04279000e-01	2.22906302e-14	4	1.36945252e+00	6.24301387e+00	5.08277295e-01	4.61182169e-01	2.59959908e+00
      513	6.76373000e-01	2.94395671e-14	4	7.95792443e-01	3.12512895e+00	2.54551535e-01	2.30385857e-01	1.29924297e+00
      2049	2.68040700e+00	3.10962928e-14	4	4.27322318e-01	1.56299916e+00	1.27306335e-01	1.15160989e-01	6.49478030e-01
      8193	1.07102850e+01	6.18987431e-14	4	2.21928115e-01	7.81552381e-01	6.36559841e-02	5.75762203e-02	3.24699627e-01
      32769	3.70961940e+01	1.60382855e-12	4	1.13227147e-01	3.90782682e-01	3.18283938e-02	2.87875647e-02	1.62344925e-01
      131073	1.54400105e+02	2.48277275e-13	3	5.72104169e-02	1.95392145e-01	1.59142949e-02	1.43937138e-02	8.11722399e-02
      524289	6.67412778e+02	4.95935021e-13	4	2.87586432e-02	9.76961728e-02	7.95720714e-03	7.19684829e-03	4.05861719e-02
      2097153	3.04734155e+03	9.91062261e-13	4	1.44182099e-02	4.88480989e-02	3.97869185e-03	3.59842307e-03	2.02931046e-02
    }\loadeduniform;
    \addplot table[x=ndof,y=eta] {\loadeduniform};
    \addplot table[x=ndof,y=energyDiff] {\loadeduniform};
    \addplot table[x=ndof,y=vmax] {\loadeduniform};
    \addplot table[x=ndof,y=errorEnergy] {\loadeduniform};
    \addplot table[x=ndof,y=errorHdiv] {\loadeduniform};
    %
    \drawslopetriangle{0.5}{1e4}{1.5e-2};
    %
    \addplot[mark=none, dashed, HUgreen, domain=1:1e7, samples=2] {0.172394989199832};
    \legend{
      {\(\eta_\ell\)},
      {\(\sqrt{E(u_\ell) - E(u)}\)},
      {\(\Vert \gradNC y_\ell \Vert_{L^\infty(\Omega)}\)},
      {\(\vvvert u - u_\ell \vvvert\)},
      {\(\Vert p - p_\ell \Vert_{H(\Div,\Omega)}\)},
      {\(0.17239892\)}
    };
  \end{loglogaxis}
\end{tikzpicture}
  \caption{Convergence history for a sequence of uniform
  triangulations of the square domain with exact solution \(u\) from
  Subsection~\ref{sec:square}.}
  \label{fig:square_convergence}
\end{figure}

With the Friedrichs constant \(C_\Friedrichs = \sqrt{2}/\pi\) of
the square domain, the criterion~\eqref{eq:uniqueness_criterion} is equivalent
to \(\Vert \gradNC v_h \Vert_{L^\infty(\Omega)} < \gamma_1^2 \operatorname{Lip}(\Grad\sigma)^{-1} (1+C_\Friedrichs^2)^{-1}
= 0.17239892\) and so
Figure~\ref{fig:square_convergence} shows that the criterion~\eqref{eq:uniqueness_criterion} holds for each level \(\ell \geq 6\) and
Theorem~\ref{thm:a_post_unique} implies global uniqueness of the computed \((u_\ell, p_\ell)\). This proves that there exists only one local minimizer in the discrete problem \eqref{eq:dpg}.

\subsection{Adaptive mesh-refinement}
\label{sec:adaptive_algorithm}

The natural adaptive algorithm with collective D\"orfler marking
\cite{MR1393904} utilizes the local error estimator
\changed{\(\eta^2(\Tri, T)\coloneq \NormLz{( I_{n\times n}+S_0) ^{-1/2}\big(\Pi_0
p_\RT - \sigma( \nabla  u_C)  + H_0 f\big)}{T}^2+ \NormLz{\Pi_0 f + \Div
p_\RT}{T}^2 + \Vert h_T f \Vert_{L^2(T)}^2\)} for
any \((u_C,p_\RT)\in S^1_0(\Tri)\times RT_0(\Tri)\) and \(T\in\Tri\) as
follows.
\medskip

\Input Regular triangulation \(\Tri_0\) of the polygonal domain \(\Omega\)
into simplices.
\begin{description}[leftmargin=\widthof{\bfseries for {}}]
  \item[\bfseries for]
    any level \(\ell=0,1,2,\dots\) \Do\\
    \Keyword{Solve} generalized LS-FEM with
    respect to triangulation \(\Tri_\ell\) and solution
    \((u_\ell, p_\ell)\).\\
    \Keyword{Compute} error estimator \(\eta_\ell \coloneq
    \eta(\Tri_\ell)\).\\
    \Keyword{Mark} a subset \(\Marked_\ell \subseteq \Tri_\ell\) of
    (almost) minimal cardinality \(\vert \Marked_\ell \vert\) with
    \[
      0.3\,\eta_\ell^2
      \leq \eta_\ell^2(\Marked_\ell) \coloneq
      \sum_{T \in \Marked_\ell} \eta^2(\Tri_\ell, T)
    \]
    \Keyword{Compute} smallest regular refinement \(\Tri_{\ell+1}\) of
    \(\Tri_\ell\) with \(\Marked_\ell \subseteq
    \Tri_\ell\setminus\Tri_{\ell+1}\) by newest-vertex bisection (NVB). \Od
\end{description}
\Output Sequence of discrete solutions \((u_\ell,
p_\ell)_{\ell\in\N_0}\) and triangulations \((\Tri_\ell)_{\ell\in\N_0}\).
\medskip

See \cite{MR2353951} for details on adaptive mesh-refinement and NVB.

\subsection{Numerical example on L-shaped domain}
\label{sec:lshape}
This subsection considers \(f \equiv 1\) on the L-shaped domain \(\Omega
\coloneq (-1,1)^2\setminus [0,1]^2 \subset \R^2\) with homogeneous
Dirichlet boundary data \(u\vert_{\partial \Omega} \equiv 0\) and unknown
exact solution \(u\).
Figure~\ref{fig:lshape_solution} displays the corresponding discrete
solutions \(u_h\) on a uniform triangulation of \(\Omega\) for the
different functions \(\phi\) from Example~\ref{ex:density}.a and b.
\begin{figure}
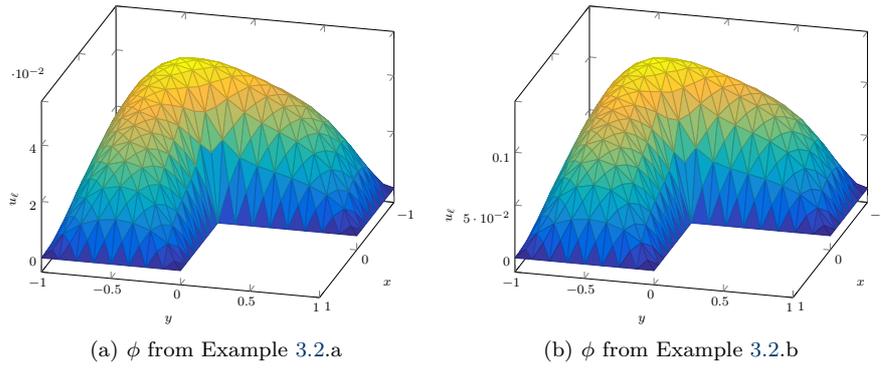

  \centering
  \subfloat[\(\phi\) from Example~\ref{ex:density}.a]{
    \input{gls_lshape_uniform_solution_Ex_a.tex}
  }
  \subfloat[\(\phi\) from Example~\ref{ex:density}.b]{
    \input{gls_lshape_uniform_solution_Ex_b.tex}
  }
  \caption{Solution \(u_h\) for different functions
  \(\phi\) on a uniform triangulation of the L-shaped domain into \(768\)
  (\(\ndof = 1\,537\)).}
  \label{fig:lshape_solution}
\end{figure}

Figure~\ref{fig:lshape_mesh} shows two typical adaptively generated
triangulations with considerable refinement at the re-entrant corner for
different functions \(\phi\).
At first glance, the meshes appear similar and resemble the undisplayed adaptive triangulation from the Poisson model problem.
\begin{figure}
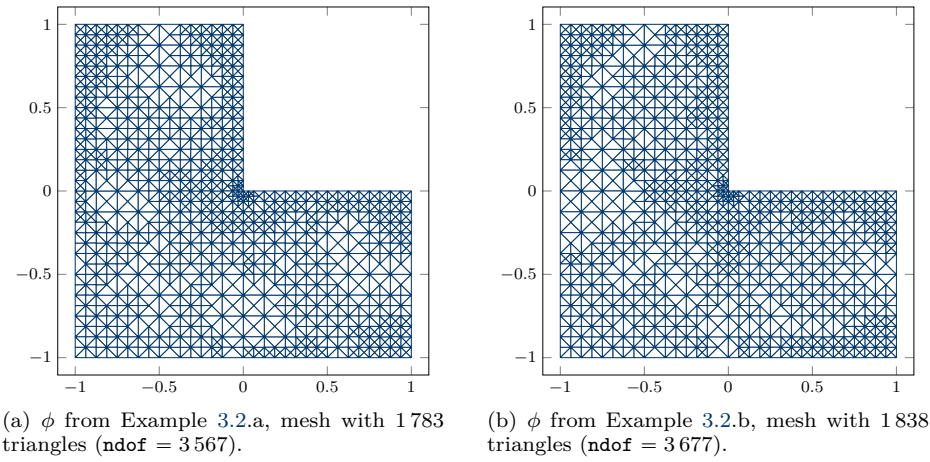

  \subfloat[\(\phi\) from Example~\ref{ex:density}.a, mesh with \(1\,783\)
  triangles (\(\ndof=3\,567\)).]{
    \input{gls_lshape_triangulation_Ex_a.tex}
  }
  \hfill
  \subfloat[\(\phi\) from Example~\ref{ex:density}.b, mesh with \(1\,838\)
  triangles (\(\ndof=3\,677\)).]{
    \input{gls_lshape_triangulation_Ex_b.tex}
  }
  \caption{Adaptively refined triangulation \(\Tri_\ell\) for different
  functions \(\phi\).}
  \label{fig:lshape_mesh}
\end{figure}

For \(\phi\) from Example~\ref{ex:density}.a,
Figure~\ref{fig:lshape_convergence} shows the convergence history plot of
the natural least-squares error estimator \(\eta_\ell=\eta(\Tri_\ell)\) and the difference of the energy
\(E(u_\ell)\) of the solution \(u_\ell\) and a reference energy
\(E(u) = -3.657423002939 \times 10^{-2}\) (computed
by the energies of \(P_1\)-conforming finite element solutions with an
Aitken \(\Delta^2\) extrapolation).

The eigenvalues of \eqref{eq:gevp} in all experiments satisfy \(1.787\leq \lambda_{\textup{min}}\leq 1.914\) and \(16.682\leq \lambda_{\textup{max}}\leq 17.932\) and so prove that all the discrete solutions are local minimizers.
The function \(\phi\) from Example~\ref{ex:density}.b leads to (undisplayed)
similar results.

For the L-shaped domain, the smallest eigenvalue \(\lambda_1 = 9.6397238\) of the
Laplacian with homogeneous Dirichlet boundary conditions yields the Friedrichs constant
\(C_\Friedrichs = 1/\sqrt{\lambda_1} = 0.32208293\).
Since \(\Vert \gradNC v_\ell \Vert_{L^\infty(\Omega)} \geq
\gamma_1^2 \operatorname{Lip}(\Grad\sigma)^{-1} (1+C_\Friedrichs^2)^{-1}
= 0.22650326\) for all level \(\ell \in \N_0\) in
Figure~\ref{fig:lshape_convergence},
Theorem~\ref{thm:a_post_unique} is \emph{not} applicable to any triangulation
\(\Tri_\ell\) of the computation at hand.

To guarantee optimal convergence rates for least-squares FEMs
\changed{with} an
alternative a~posteriori error estimator, the choice of a sufficiently
small bulk parameter is crucial \cite{MR3296614,CR16}. However, for
the natural error estimator with the values of the least-squares
functional, the plain convergence proof of \cite{CPB16} requires the bulk
parameter sufficiently close to \(1\).
For the nonlinear model problem at hand, the convergence history plot in Figure~\ref{fig:lshape_convergence} provides numerical evidence for optimal
convergence rates for adaptive mesh-refinement of Subsection~\ref{sec:adaptive_algorithm}
and suboptimal convergence for uniform refinement.
\begin{figure}
  \begin{tikzpicture}[scale=0.85]
  \begin{loglogaxis}[convergenceplot, xmin=10, xmax=2.5e6]
    \pgfplotsset{cycle list={
      {densely dotted, HUblue,  mark=square*,   every mark/.append style={solid, fill=black!25!white}},%
      {densely dotted, HUred,   mark=*,         every mark/.append style={solid, fill=black!25!white}},%
      {densely dotted, HUgreen, mark=triangle*, every mark/.append style={solid, fill=black!25!white}},%
      {HUblue,  mark=square*,   every mark/.append style={fill=HUblue!60!white}},%
      {HUred,   mark=*,         every mark/.append style={fill=HUred!60!white}},%
      {HUgreen, mark=triangle*, every mark/.append style={fill=HUgreen!60!white}},%
    }};
    \pgfplotstableread{
      ndof	runtime	acc	it	vmax	eta	energyDiff
      25	4.60418000e-01	5.23583411e-14	4	3.02521008e-01	1.77434343e+00	1.50185837e-01
      97	2.14552000e-01	9.25977240e-16	4	2.41428368e-01	8.99841184e-01	8.56997678e-02
      385	4.93834000e-01	1.24341870e-15	4	2.62655884e-01	4.54839693e-01	4.69391549e-02
      1537	1.88351900e+00	2.81073788e-15	4	3.05021958e-01	2.30244699e-01	2.58594718e-02
      6145	7.55094600e+00	5.51128023e-15	4	3.62018877e-01	1.17121260e-01	1.45470254e-02
      24577	3.16183030e+01	1.10428650e-14	4	4.44130608e-01	6.00522691e-02	8.38813932e-03
      98305	1.37740859e+02	2.18956913e-14	4	5.49122671e-01	3.11417156e-02	4.94832326e-03
      393217	5.79547617e+02	4.37056039e-14	4	6.80504833e-01	1.63956006e-02	2.97139675e-03
      1572865	2.45713386e+03	8.77440255e-14	4	8.44255461e-01	8.79738605e-03	1.80314013e-03
    }\loadeduniform;
    \pgfplotstableread{
      ndof	runtime	acc	it	vmax	eta	energyDiff
      25	4.61659000e-01	5.23583411e-14	4	3.02521008e-01	1.77434343e+00	1.50185837e-01
      73	1.90479000e-01	6.17467819e-16	4	2.39229919e-01	1.16862623e+00	8.65876998e-02
      93	1.34841000e-01	7.67899381e-16	3	2.41877768e-01	9.33705890e-01	8.61643460e-02
      197	2.61637000e-01	8.91450268e-16	4	2.61221690e-01	7.23349583e-01	6.57856492e-02
      305	3.26940000e-01	1.67484869e-12	4	2.60379617e-01	5.36600483e-01	5.15232517e-02
      507	5.11429000e-01	1.84981567e-13	3	2.66596870e-01	4.22113425e-01	4.44501141e-02
      825	1.00202400e+00	2.04459591e-15	4	3.05252923e-01	3.47965928e-01	3.27712915e-02
      1113	1.11171800e+00	5.74660276e-15	3	3.04178977e-01	2.90703984e-01	2.90882797e-02
      1695	2.11022600e+00	3.16499027e-15	4	3.62431062e-01	2.37688781e-01	2.29063587e-02
      2253	2.23427500e+00	3.29360532e-15	3	3.62878682e-01	2.01608293e-01	2.08407534e-02
      3567	4.36658400e+00	4.45740486e-15	4	4.46388940e-01	1.67820352e-01	1.55952370e-02
      4635	4.70117300e+00	5.18318010e-15	3	4.46367957e-01	1.42142111e-01	1.41184709e-02
      6901	8.64875700e+00	6.24182124e-15	4	5.52946641e-01	1.17619757e-01	1.13397681e-02
      8961	1.13425000e+01	7.14108324e-15	4	6.86704061e-01	1.00859909e-01	1.01823122e-02
      13851	1.76438030e+01	9.49877233e-15	4	8.52682708e-01	8.55511932e-02	7.95780818e-03
      17903	1.88043940e+01	1.06315632e-14	3	8.52686548e-01	7.32109295e-02	7.10088129e-03
      24033	3.11965900e+01	1.26318388e-14	4	1.05784946e+00	6.26634286e-02	5.85909448e-03
      32845	4.34223100e+01	1.47341829e-14	4	1.31792446e+00	5.34587492e-02	5.27685704e-03
      46037	6.10248500e+01	1.76907565e-14	4	1.63939564e+00	4.61959209e-02	4.28643901e-03
      64923	7.05435890e+01	2.11031685e-14	3	1.63929669e+00	3.91435702e-02	3.67164872e-03
      78045	8.60128990e+01	2.15027929e-14	3	1.63931944e+00	3.40002515e-02	3.45366858e-03
      114345	1.28115024e+02	4.48731729e-13	3	2.04414127e+00	2.90478152e-02	2.76085913e-03
      142623	1.61998489e+02	1.38875424e-13	3	2.56032857e+00	2.52692023e-02	2.58588063e-03
      213023	2.60959973e+02	5.33946286e-14	3	3.20222278e+00	2.17693245e-02	2.01213375e-03
      275245	2.79274632e+02	2.67048646e-12	3	3.20222691e+00	1.87942073e-02	1.79409184e-03
      346123	4.45504583e+02	5.14357336e-14	3	4.01227609e+00	1.63757624e-02	1.50183957e-03
      484671	6.32331524e+02	6.03663576e-14	3	5.04878301e+00	1.41360007e-02	1.34247863e-03
      595711	7.77882150e+02	6.90915811e-14	3	6.34879547e+00	1.23624475e-02	1.19313005e-03
      884627	1.17729788e+03	8.08823960e-14	3	6.34235067e+00	1.06816361e-02	9.70694639e-04
      1114109	1.54366993e+03	8.62715091e-14	3	6.34232799e+00	9.30462641e-03	8.73111507e-04
    }\loadedadaptive;
    \addplot table[x=ndof,y=eta] {\loadeduniform};
    \addplot table[x=ndof,y=energyDiff] {\loadeduniform};
    \addplot table[x=ndof,y=vmax] {\loadeduniform};
    \addplot table[x=ndof,y=eta] {\loadedadaptive};
    \addplot table[x=ndof,y=energyDiff] {\loadedadaptive};
    \addplot table[x=ndof,y=vmax] {\loadedadaptive};
    %
    \addplot[mark=none, dashed, HUgreen, domain=1:1e7, samples=2] {0.226503256736024};
    \drawslopetriangle{0.5}{3e4}{1e-3};
    \drawswappedslopetriangle{0.4}{3e5}{1e-2};
    \legend{,,,%
      {\(\eta_\ell\)},
      {\(\sqrt{E(u_\ell) - E(u)}\)},
      {\(\Vert \gradNC y_\ell \Vert_{L^\infty(\Omega)}\)},
      {\(0.22650326\)}
    };
  \end{loglogaxis}
\end{tikzpicture}
  \caption{Convergence history for adaptive mesh-refinement (solid lines)
    and uniform mesh-refinement (dotted lines)
    with \(\phi\) from Example~\ref{ex:density}.a.}
  \label{fig:lshape_convergence}
\end{figure}
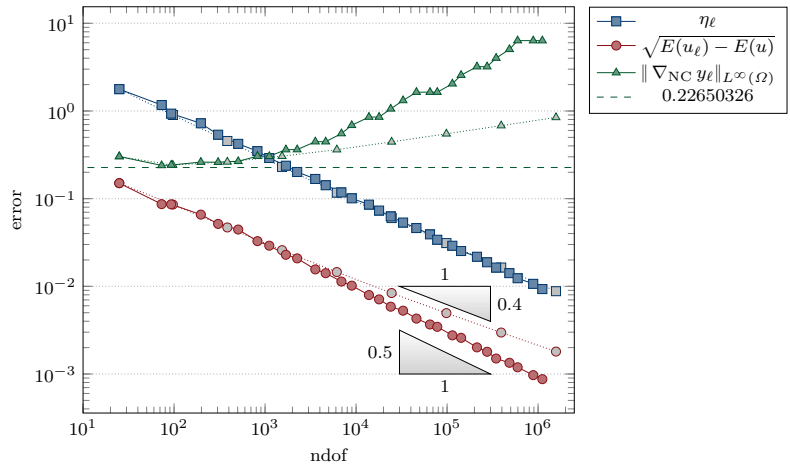


\bibliography{PaperALC}

\end{document}